%% file: ex_article.tex
\documentclass[review,onefignum,onetabnum]{siamart171218}

\usepackage{geometry}
\usepackage{commath}
\usepackage{tcolorbox}
\usepackage{amssymb}

\input{ex_shared}

\ifpdf
\hypersetup{
  pdftitle={Parameterized Wasserstein Hamiltonian Flow},
  pdfauthor={robot}
}
\fi

\myexternaldocument{ex_supplement}

\newcommand{\deltaT}{\frac{\delta}{\delta T}}
\newcommand{\deltarho}{\frac{\delta}{\delta \rho}}
\newcommand{\STinv}{S \circ T^{-1}}
\newcommand{\TSeps}{T_{\varepsilon}^S}
\newcommand{\rhoeps}{\rho_{\varepsilon}}
\newcommand{\Fcal}{\mathcal{F}}

\begin{document}

\maketitle

\begin{abstract} In this paper, we propose a new method to compute the solution of time-dependent Schr\"odinger equation (TDSE). Using push-forward maps and Wasserstein Hamiltonian flow, we reformulate the TDSE as a Hamiltonian system in terms of push-forward maps. The new formulation can be viewed as a generative model in the Wasserstein space, which is a manifold of probability density functions. Then we parameterize the push-forward maps by reduce-order models such as neural networks. This induces a new metric in the parameter space by pulling back the Wasserstein metric on density manifold, which further results in a system of ordinary differential equations (ODEs) for the parameters of the reduce-order model. Leveraging the computational techniques from deep learning, such as Neural ODE, we design an algorithm to solve the TDSE in the parameterized push-forward map space, which provides an alternative approach with the potential to scale up to high-dimensional problems. Several numerical examples are presented to demonstrate the performance of this algorithm. 
\end{abstract}


\section{Introduction}
Schr\"odinger equation plays a fundamental role in the study of quantum physics. In this paper, we are concerned with its numerical simulation. To better explain our objectives and ideas, we take the following nonlinear time-dependent Schr\"odinger equation (TDSE) as an example:
\begin{align}
    \label{def: TDSE}
    i\frac{\partial}{\partial t}\psi (t, x)=-\frac{1}{2}\Delta \psi(t, x) + \frac{\delta}{\delta \rho}\mathcal{F}_R(|\psi(t, x)|^2, x)\cdot \psi(t, x),
\end{align}
where $\psi$ is a complex-valued function defined on $[0,T] \times \mathbb{R}^d$, $\rho=|\psi|^2$, which can be viewed as a probability density function associated with $\psi$, $\mathcal{F}_R$ is a functional of $\rho$, and $\frac{\delta}{\delta \rho}\mathcal{F}_R$ is the $L^2$ first variation of $\mathcal{F}_R$.

With different choices of $\mathcal{F}_R$, the TDSE \eqref{def: TDSE} models various physical problems, for example, quantum harmonic oscillator \cite{dekker1981classical}, many particle interaction systems \cite{Dirac1929QuantumMO}, and Bose-Einstein condensation \cite{doi:10.1126/science.269.5221.198, alma991012567879703276}, among many others. There exists an extensive literature with many remarkable advancements on the theory, computation, and application of \eqref{def: TDSE}. However, its direct numerical simulation remains a difficult task, especially when the dimension $d$ is high, e.g., $d \ge 4$. To mitigate the computational challenges, we introduce a novel formulation that takes advantage of the most recent developments in generative models \cite{goodfellow2014generative} from machine learning and Wasserstein Hamiltonian flow (WHF) \cite{chow2020wasserstein} related to optimal transport theory \cite{villani2009optimal}. More specifically, there are two main objectives in this paper:
\begin{enumerate}
    \item Reformulate the TDSE \eqref{def: TDSE}, via the Madelung transform, as a generative model using push-forward maps in conjunction with a WHF. 
    \item Propose a numerical method that applies parameterized reduced-order models, such as deep neural networks (DNNs), to solve the generative model. The result is a system of ordinary differential equations (ODEs) in the parameter space, which can be solved by symplectic numerical schemes. 
\end{enumerate}

Unlike the existing formulations of TDSEs, the derived generative model provides a dynamical description of the push-forward maps. The induced numerical method is complementary to existing approaches and can be used as an alternative algorithm that is cost-efficient for high-dimensional simulations.

In Section \ref{sec: relatedwork}, we give a brief discussion of related work in the literature. We provide concise introductions to WHF and generative models, two crucial tools used in this investigation, in Section \ref{sec: tools}. The reformulation of TDSE into a generative model is presented in Section \ref{sec: TDSE}. The numerical method based on the formation of neural networks is introduced in Sections \ref{sec: parameterized ge} and \ref{sec: numerical method}. We illustrate the performance of the method by several examples in Section \ref{sec:example}, followed by a short discussion to conclude the paper. 




\section{Related work} \label{sec: relatedwork}
Numerical simulations of TDSE have been conducted extensively with numerous algorithms based on classical methods such as finite difference \cite{ doi:10.1137/S1064827594277053, 10.1119/1.1973991, PhysRevB.43.6760, 10.1063/1.1753661}, spectral method 
\cite{WeizhuB, doi:10.1137/S1064827501393253}, and level set methods \cite{jin2005computing}. A survey of different numerical schemes with 
comprehensive comparisons on their properties can be found in \cite{antoine2013computational}. Classical methods can provide efficient and accurate solutions when the dimension is small, i.e. $d \leq 3$. However, they suffer a serious issue known as the curse of dimensionality, referring to the exponential growth in the computational cost with respect to the dimension $d$, especially when $d$ is large. Compared to classical methods, particle dynamics-based simulations such as the quantum trajectory method (QTM) \cite{10.1063/1.1753661} and smooth particle hydrodynamics (SPH) \cite{PhysRevE.91.053304} have been proposed. They are Monte Carlo approaches designed using the formulation of Bohmian mechanics \cite{Berndl1995OnTG,durr2009bohmian,PhysRevLett.86.3215}, which is a major source of motivations for our investigation. For example, the SPH generates samples from the initial distribution and simulates the sample dynamics under the quantum potential. These sample-based approaches scale well with dimensions, but the approximation to the density function as well as its gradient evaluation remain a major challenge in the computation. 

In recent years, machine learning-based methods, such as the physcis-informed neural network (PINN) \cite{raissi2019physics}, deep Ritz method (DRM) \cite{e2018deep}, weak adversarial network (WAN) \cite{zang2020weak}, and many others \cite{han2017deep, ruthotto2020machine} have shown promising results of using neural networks to solve partial differential equations (PDEs) in high dimensions. Some of them have been adopted to solve Schr\"odinger equation. For example, PINN has been used in computing the solution of TDSE \cite{Pu2021SolvingLW, Wu2022PhysicsIR}. A specific class of neural networks has been proposed to represent the many-electron wave function following the Pauli exclusion principle \cite{HAN2019108929, hermann2020deep, pfau2020ab}, which shows powerful and promising results for estimating ground state energies in many-electron systems. More research has been reported on the application of neural networks to simulate many body problems \cite{doi:10.1126/science.aag2302} and molecular dynamics \cite{lu202186}. Those studies are among the other sources that motivated this work.

Recently, a method called Deep Stochastic Mechanics (DSM) \cite{orlova2023deep}, which modifies the formulations in \cite{nelson1966derivation} and \cite{guerra1983quantization}, was developed to generate samples following the time-evolving squared magnitude (density) of the solution to TDSE. While Madelung transform induces the coupled evolution PDEs of the phase and logarithm of density, the gradients of them satisfy a new coupled PDEs, which are solved by the PINN approach in DSM. These gradients can add up to form the drift of a stochastic differential equation (SDE), whose sample trajectories follow the density of TDSE. In contrast, we directly use the reformulation based on the Madelung transform and approximate the push-forward map by solving an ODE system which is free of network training.

Our formulation and algorithm are directly inspired by recent advancements in parameterized WHF \cite{wu2025parameterized}, in which a numerical method based on neural network parameterization has been proposed for WHF in conjunction with generative models. It is known that, by using Madelung transform, the TDSE can be rewritten as a continuity equation coupled with a Hamilton-Jacobi equation, which forms a WHF. Hence, we can apply the parameterized WHF to solve the TDSE, providing substantial new improvements to handle the Fisher information that appeared in the Hamilton-Jacobi equation. To this end, we borrow a tool from generative models called Neural ODE \cite{chen2018neural} that can compute Fisher information efficiently.

\section{Mathematical Tools} \label{sec: tools} In this section, we give concise introductions to the WHF and Neural ODE, which are two main tools used to establish the reformulation of TDSE as generative models. 

\subsection{Wasserstein Hamiltonian Flow} Let $M$ be a smooth manifold without boundary. For simplicity, we can assume $M = \mathbb{R}^d$ in our discussion. We consider the set of density functions defined on $M$ with bounded second moment:
\begin{align}
    \label{def: prob manifold}
    \mathcal{P}(M)=\cbr[2]{\rho\in C^{\infty}(M)\, : \, \rho\geq 0,\ \int_M\rho  \,dx=1,\ \int_M |x|^2\rho \,dx < \infty }.
\end{align}


Given a Hamiltonian $\mathcal{H}(\rho, \Phi)= \frac{1}{2}\int_M |\nabla \Phi|^2\rho(x)dx+\mathcal{F}(\rho)$ where $\Phi \in C^{\infty}$ and $\mathcal{F}$ is an energy functional on $\mathcal{P}(M)$, the WHF with Hamiltonian $\mathcal{H}$ is described by the following system,
\begin{subequations}
\label{def: whf}\begin{align}
    &\partial_t\rho = \frac{\delta}{\delta\Phi}\mathcal{H}(\rho, \Phi),\\
    &\partial_t\Phi = -\frac{\delta}{\delta\rho}\mathcal{H}(\rho, \Phi),
\end{align}
\end{subequations}
with initial values
\begin{equation}
    \label{eq:WHF-init}
    \rho(0, x)=\rho_0(x) \qquad \mbox{and} \qquad \Phi(0, x)=\Phi_0(x).
\end{equation}

Many classical PDEs can be formulated as WHF with different energy functionals $\mathcal{F}$ \cite{chow2020wasserstein}. The
Schr\"odinger equation is among them. More details are given next.  



\subsection{Schr\"odinger equation and Madelung transform} 
Consider the Madelung transform 
\begin{align}
    \label{def: madelung transform}
    \psi(t, x)=\sqrt{\rho(t, x)}e^{i\Phi(t, x)}, 
\end{align}
which is a nonlinear symplectic transform mapping a complex-valued wave function to the pair of real-valued functions $(\rho, \Phi)$, see \cite{reddiger2023towards} for more details. Using the Madelung transform, the Schr\"odinger equation \eqref{def: TDSE} can be reformulated as the following Madelung system by plugging \eqref{def: madelung transform} into \eqref{def: TDSE} and matching the real and imaginary parts on its two sides:
\begin{subequations}
\label{eq: madelung}
\begin{align}
&\partial_t \rho + \nabla\cdot(\rho\nabla\Phi) = 0, \\ 
& \partial_t \Phi +\frac{1}{2}|\nabla\Phi|^2 = - \frac{\delta}{\delta\rho}\mathcal F_R(\rho) + \frac12\frac{\Delta\sqrt{\rho}}{\sqrt{\rho}}, 
\end{align}
\end{subequations}
with initial values $\rho(0,x) = |\psi(0, x)|^2$ and $\Phi(0,x)$ being the phase of $\psi(0,x)$ at every $x \in M$.
Due to the Madelung transform, we assume $\rho>0$ almost everywhere in $M$ for any time $t$ hereafter.

\begin{remark}
  We always assume that $\Phi$ in \eqref{eq: madelung} is single-valued, thereby eliminating the potential inequivalence between the Schrödinger equation and the Madelung system \cite{wallstrom1994inequivalence}. For more detailed discussions, we refer the reader to \cite{wallstrom1994inequivalence,orlova2023deep} and the references therein.
\end{remark}

From the perspective of Lagrangian mechanics on the density manifold $\mathcal P(M)$, the Schr\"odinger equation arises as the critical point of an action functional \cite{lafferty1988density}. Furthermore, an optimal transport approach was proposed in \cite{von_Renesse_2012} to interpret the Madelung system \eqref{eq: madelung} as the Hamiltonian flow \eqref{def: whf} associated with the following Hamiltonian:
\begin{align}
    \label{def: se hamiltonian}
    \mathcal{H}(\rho, \Phi) = \int_{\mathbb{R}^d} \frac{1}{2}|\nabla \Phi(x)|^2\rho(x) dx+\mathcal{F}_R(\rho)+\frac{1}{8}\mathcal{F}_Q(\rho),
\end{align}
where the last term $\mathcal{F}_Q=\int_{\mathbb{R}^d}\lvert \nabla \log\ \rho(x)\rvert^2\rho(x)dx$ is known as the Fisher information in the context of quantum mechanics \cite{reginatto1998derivation}, and its $L^2$ first variation is the quantum potential \cite{bohm1952suggested12}:
\begin{align}
    \label{eq: fi variation}
    \frac{\delta}{\delta\rho}\mathcal{F}_Q(\rho, x)=-2\Delta \log \rho-\lvert \nabla \log \rho\rvert^2 = - 4\frac{\Delta\sqrt{\rho}}{\sqrt{\rho}}. 
\end{align}
Our reformulation of \eqref{def: TDSE} as a generative model is based on its WHF formulation \eqref{eq: madelung}. 

\subsection{Neural ODE} Generative models in machine learning include a large class of algorithms that aim to map samples from an initial distribution with density $\lambda$ to samples of a target distribution with density $\mu$. It is often easy to obtain samples from the initial distribution such as Gaussian while difficult to sample from the target one, about which only partial information is accessible. Well-known examples of generative models include generative adversarial nets (GAN) \cite{goodfellow2014generative}, Wasserstein GAN (WGAN) \cite{arjovsky2017wasserstein}, and diffusion generative model \cite{song2021score}. A key component of a generative model is a push-forward map $T:\mathbb{R}^d \rightarrow \mathbb{R}^d$ defining a correspondence between samples $z \rightarrow T(z)$ where $z$ is sampled from $\lambda$, denoted by $x \sim \lambda$. The map $T$ induces a push-forward distribution with density $T_\sharp \lambda$ through
\begin{align*}
    \int_E T_{\sharp}\lambda(x)\,dx = \int_{T^{-1}(E)} \lambda(z) \, dz \quad \textrm{ for any measurable set\ } E\subset \mathbb{R}^d,
\end{align*}
where $T^{-1}(E)$ is the pre-image of $E$.
Correspondingly, we also have the change of variable formula to explicitly evaluate the induced push-forward density:
\begin{align}
    \label{eq: change of vari}
    T_{\sharp}\lambda(z)=\lambda\circ T^{-1}(z)\,\textrm{det} \Big(\frac{d}{dz}T^{-1}(z) \Big)\quad \forall z\in \mathbb{R}^d.
\end{align}
In a generative model, $T$ is realized by a neural network and it is desirable to achieve $T_\sharp \lambda = \mu$. 

To reformulate the Schr\"odinger equation \eqref{def: TDSE} in a generative model, we select a special type of neural network called 
Neural Ordinary Differential Equations (Neural ODE) as the push-forward map $T$. Neural ODE was first proposed in \cite{chen2018neural}. It is defined through the solution map of a parameterized ODE. More precisely, let $[0, t^*]$ be a fixed time interval and  $f_{\theta}: \mathbb{R}^d\rightarrow\mathbb{R}^d$ be a neural network with parameters $\theta\in \Theta \subset \mathbb{R}^{m}$ ($m$ is the number of parameters in $f_{\theta}$), the push-forward map given by the neural ODE is defined as $T_\theta(z)=w(t^*)$, where $w(t)$ satisfies
\begin{equation}
\label{def: node forward}
\dot w(t) =f_{\theta}(w(t)),\quad w(0)=z, 
\end{equation}

It is shown in numerous studies that Neural ODE has many desirable properties like strong approximation power and efficient implementations. More importantly, we select Neural ODE because it is easily invertible due to the uniqueness of the solution for the ODE \eqref{def: node forward}, and the logarithmic density function can be conveniently evaluated through the adjoint equation \cite{chen2018neural}. Both features are crucial for us to express \eqref{def: TDSE} as a continuous-time push-forward map and compute the Fisher information term appeared in  \eqref{eq: fi variation}. A brief discussion on how to implement the Fisher information term through Neural ODE is provided in Appendix \ref{sec: evaluation of fi}.


\section{Expressing the Schr\"odinger equation by push-forward map} \label{sec: TDSE}
In this section, we present the reformulation of the TDSE \eqref{def: TDSE} in terms of time-dependent push-forward map $T$. 


Let $\lambda >0$ be the density of a reference distribution on $\mathbb{R}^d$. We define \begin{align}
    \mathcal{O}=\{T\in L^2(\mathbb{R}^d;\mathbb{R}^d, \lambda): ~T\textrm{ is diffeomorphism from }\mathbb{R}^d\textrm{ to }\mathbb{R}^d.\}
\end{align} 
as the space of diffeomorphisms. We denote $\mathcal{T}_T\mathcal{O}:=C^\infty(\mathbb{R}^d; \mathbb{R}^d) \bigcap L^2(\mathbb{R}^d;\mathbb{R}^d, \lambda)$ as its tangent space at $T$. For any $T\in \mathcal{O}$ and $\sigma_i \in \mathcal{T}_T\mathcal{O}$ ($i=1,2$), we define a metric $\mathcal{G}$ on the tangent bundle $\mathcal{T}\mathcal O$ as
\begin{align}
    \label{def: op metric}
    \mathcal{G}(T)(\sigma_1, \sigma_2)=\int \sigma_1(z)^\top \sigma_2(z)\lambda(z)~dz. 
\end{align}
Then, we introduce a Lagrangian $\mathbb{L}$ defined on $\mathcal T \mathcal O$ as
\begin{align} \label{def: Lagrangian}
    \mathbb L(T, \sigma)=\frac{1}{2}\mathcal{G}(T)(\sigma, \sigma)-\mathcal{F}(T_{\sharp}\lambda).
\end{align}
Now, for a time-dependent curve $\{T_t\}$ ($0\leq t\leq t_0$) on $\mathcal{O}$, we denote $\dot T_t = \frac{d T_t}{dt} \in \mathcal{T}_{T_t}\mathcal{O}$ as the time derivative of $T_t$ at $t$, and denote $\mathcal{F}(\rho) = \mathcal{F}_R (\rho)+ \frac{1}{8}\mathcal{F}_Q (\rho)$. In order to introduce the Hamiltonian flow on $\mathcal O$, we first consider the stationary point associated with the following action functional $\mathcal I:\mathcal O_{0, t_0} \rightarrow \mathbb{R}$
\begin{align}
    \label{def: op vari}
    \mathcal{I}(\{T_t\}) = \int_0^{t_0} \mathbb{L}(T_t, \dot T_t)\,dt, \quad \{T_t\} \in \mathcal O_{0, t_0}, 
\end{align}
where we define $\mathcal O_{0, t_0}$ as the set of smooth paths $\{T_t\}_{0\leq t \leq t_0}$ on $\mathcal O$ with $T_{0 \sharp}\lambda = \rho_0$ and $T_{t_0 \sharp}\lambda = \rho_{t_0}$. Here, $\rho_0, \rho_{t_0}$ are the density functions obtained by the Madelung transform \eqref{eq: madelung} from the wavefunction $\psi(t,x)$ at $t=0$ and $t=t_0$ respectively. The stationary point associated with \eqref{def: op vari}, if it exists, satisfies the Euler–Lagrange equation, as stated in the following theorem. For brevity, we omit the subscript \( t \) and use \( T \) and \( \dot{T} \) to denote \( T_t \) and \( \dot{T}_t \), respectively, in the following discussion. 

\begin{theorem}[Hamiltonian system in the space of diffeomorphisms]    \label{theorem: whf in push forward space} The critical point of \eqref{def: op vari}, if exists, satisfies
    \begin{align}
    \label{def: op el}
        \ddot T(z)=-\nabla_X\frac{\delta}{\delta \rho}\mathcal{F}(T_{\sharp}\lambda(\cdot), \cdot)\circ T(z). 
    \end{align}
    Equivalently, by introducing the momentum $\Lambda:=\frac{\partial \mathbb{L}(T, \dot T)}{\partial \dot T} = \lambda\dot{T}$, equation \eqref{def: op el} can also be written as a first order Hamiltonian system
    \begin{equation}
        \label{eq: op hf}
        \begin{split}
            &\frac{d}{dt}T=\frac{1}{\lambda} \Lambda,\\
            &\frac{d}{dt}\Lambda=-\nabla_X\frac{\delta}{\delta \rho}\mathcal{F}(T_{\sharp}\lambda(\cdot), \cdot)\circ T(z) \lambda(z),
        \end{split}
    \end{equation}
    with Hamiltonian $\mathbb{H}(T, \Lambda)= \mathcal{G}(T)(\dot{T},\dot{T}) - \mathbb L(T, \dot T)=\frac{1}{2}\mathcal{G}(T)(\frac{1}{\lambda}\Lambda, \frac{1}{\lambda}\Lambda)+\mathcal{F}(T_{\sharp}\lambda)$.
Furthermore, the push-forward density $T_{\sharp}\lambda$ solves the WHF \eqref{def: whf} whose Hamiltonian is given by \eqref{def: se hamiltonian}.
\end{theorem}

\begin{proof}
    For convenience, we denote $\mathbb{F}(T):=\mathcal{F}(T_\sharp \lambda)$. Following the definition of $\mathbb{L}(T, \dot T)$ in \eqref{def: Lagrangian}, we have
    \begin{align} \label{def: Lagrangian-2}
        \mathbb L(T, \dot T)=\frac{1}{2}\mathcal{G}(T)(\dot T, \dot T)-\mathbb{F}(T),
    \end{align}    

    The Euler-Lagrange equation satisfied by the critical point of \eqref{def: op vari} is
    \begin{align}
    \label{eq: el eq}
        \frac{\delta}{\delta T}\mathbb{L}(T, \dot T)=\frac{d}{dt}\frac{\delta}{\delta \dot T}\mathbb{L}(T, \dot T).
    \end{align}
    Let us first examine its right-hand side. Following the definition given in \eqref{def: op metric}, we note that the metric  $\mathcal{G}$ is the same for all $T$, implying that $\mathcal{G}$ is independent of $T$. This leads to $\frac{\delta}{\delta \dot T}\mathbb{L}(T, \dot T)=\dot T(z)\lambda(z)$. Hence we have 
    \begin{align}
    \label{eq: el eq right}
       \frac{d}{dt}\frac{\delta}{\delta \dot T}\mathbb{L}(T, \dot T) = \ddot{T}(x) \lambda(z).
    \end{align}
    We next compute the left-hand side of \eqref{eq: el eq}. Since the metric is independent of $T$, we have
    $\frac{\delta}{\delta T}\mathbb{L} = -\frac{\delta}{\delta T}\mathbb{F}$. In the following, we calculate $\deltaT \mathbb{F}(T)$.

    Let $S \in \mathcal{O}$ be arbitrary and define $\TSeps(x)$ to be the solution $X(\varepsilon)$ of the ODE: 
    \begin{equation}
    \label{eq:proof-ode}
        \begin{cases}
        X'(t) = (\STinv)(X(t)), \quad 0 \le t \le \varepsilon, \\
        X(0) = x = T(z) \sim T_\sharp \lambda ,
        \end{cases}
    \end{equation}
    at the end time $t = \varepsilon$, where $z \sim \lambda$. In addition, $\TSeps$ is a variation of $T$ equivalent to $T+\varepsilon S$ up to the first order because for any $z$ there are $X(0)=x=T(z)$ and
    \begin{align*}
        X(\varepsilon) 
        & = X(0)+X'(0)\varepsilon + o(\varepsilon) \\
        & = x + (\STinv)(x)\varepsilon + o(\varepsilon) \\
        & = T(z) + \varepsilon S(z) + o(\varepsilon) \\
        & = (T + \varepsilon S)(z) + o(\varepsilon).
    \end{align*} 
    We call $\rho_t$ the density of $X(t)$ for $0 \le t \le \varepsilon$. Then the continuity equation of $\rho_t$ corresponding to \eqref{eq:proof-ode} is
    \begin{equation}
    \label{eq:proof-ct-eq}
        \begin{cases}
        \partial_t \rho_t(x) + \nabla_X \cdot [\rho_t(x) (\STinv)(x)] = 0,\quad 0 \le t \le \varepsilon,\\
        \rho_0 = T_\sharp \lambda.
        \end{cases}
    \end{equation}
    Now we derive the $L^2$ first-variation $\deltaT \mathbb{F}(T)$. We have
    \begin{align*}
        \Big \langle \deltaT \mathbb{F}(T), S \Big\rangle 
        & = \frac{d}{d\varepsilon} \mathbb{F}(\TSeps)\Big|_{\varepsilon = 0}= \frac{d}{d\varepsilon} \Fcal(\TSeps{}_\sharp\lambda )\Big|_{\varepsilon = 0} \\
        & = \lim_{\varepsilon \to 0} \int \deltarho \Fcal(\rho_0)(x) \frac{\rhoeps(x) - \rho_0(x)}{\varepsilon} \,dx \\
        & = \int \deltarho \Fcal (\rho_0) (x)  \partial_t \rho_\varepsilon(x) \Big|_{\varepsilon=0} \,dx.
    \end{align*}
    By the continuity equation \eqref{eq:proof-ct-eq}, we know 
    \[
    \partial_t \rho_\varepsilon(x) \big|_{\varepsilon=0} = - \nabla \cdot [\rho_0 (x) (\STinv)(x)].
    \]
    Therefore
    \begin{align*}
        \Big \langle \deltaT \mathbb{F}(T), S \Big\rangle 
        & = \int \deltarho \Fcal (\rho_0) (x)  \Big(-\nabla_X \cdot [\rho_0 (x) (\STinv)(x)] \Big) \,dx \\
        & = \int \nabla_X \deltarho \Fcal (\rho_0) (x)   (\STinv)(x) \rho_0 (x) \,dx \\
        & = \int \nabla_X \deltarho \Fcal (T_\sharp \lambda) (T(z)) S(z) \lambda (z) \,dz,
    \end{align*}
    where the last equation is due to the change of variable $x = T(z)$.
    Since $S$ is arbitrary, we must have
    \[
    \deltaT \mathbb{F}(T) = \nabla_X \deltarho \Fcal (T_\sharp \lambda) (T(z)) \lambda (z).
    \]   

This leads to 
        \begin{align*}
        \frac{\delta}{\delta 
 T}\mathbb{L}(T, \dot T)&=-\frac{\delta}{\delta T}\mathbb{F}(T)=-\nabla_X\frac{\delta}{\delta \rho}\mathcal{F}(T_{\sharp}\lambda(\cdot), \cdot)\circ T(z) \lambda(z). 
    \end{align*}
    Hence, the Euler-Lagrange equation \eqref{eq: el eq} becomes
    \begin{align}
        \frac{d}{dt}\left(\dot T(z)\lambda(z)\right)=\nabla_X\frac{\delta}{\delta \rho}\mathcal{F}(T_{\sharp}\lambda(\cdot), \cdot)\circ T(z) \lambda(z) 
    \end{align}
    Since $\lambda(z)>0$ for all $z\in \mathbb{R}^d$, this proves \eqref{def: op el}. 
    
    Following Proposition 2 in \cite{chow2020wasserstein}, we know that the push-forward density $T_{\sharp}\lambda$ satisfies the WHF \eqref{def: whf}, which gives \eqref{eq: op hf}. This completes the proof.
\end{proof}


For convenience, we call the equation \eqref{def: op el} or system \eqref{eq: op hf} the generative model reformulation of TDSE \eqref{def: TDSE}. We would like to mention that the derivation of \eqref{def: op el} is also inspired by the \textit{Bohmian mechanics} \cite{bohm1952suggested12}, 
which aims at describing the particle dynamics governed by the Schr\"odinger equation. More precisely, each particle $\boldsymbol{X}$ moves according to a dynamical system 
\begin{align}
    \label{def: guiding eq}
    \ddot{\boldsymbol{X}}=-\nabla_X\frac{\delta}{\delta\rho}\mathcal{F}(\rho, \boldsymbol{X}),
\end{align}
assisted by the guiding function $\rho$, which is obtained by the Schr\"odinger equation and Madelung transform. Clearly, equations \eqref{def: guiding eq} and \eqref{def: op el} share similarities in their formulation. However, they have significant distinctions. By definitions, we see that the system \eqref{def: op el} is defined in the space of diffeomorphisms $\mathcal{O}$ while the equation \eqref{def: guiding eq} is defined in $\mathbb{R}^n$. More importantly, the operator equation \eqref{def: op el} is self-contained. Its trajectory is fully determined if the initial conditions of $T$ and $\Lambda$ are given. This implies that it can produce particle motions as well as density maps at the same time once $T_t$ is known. In sharp contrast, \eqref{def: guiding eq} is not self-contained. It requires extra knowledge of the guiding function $\rho$ to fully determine the motions of the particles. It is worth noting that investigating the well-posedness of the operator equation \eqref{def: op el} and the Vlasov-type ODE \eqref{def: guiding eq} lies beyond the scope of this study and is left as a potential direction for future research.

To further digest the system \eqref{eq: op hf}, we consider its implication in the \textit{energy eigen-state}. In this case, the generative model \eqref{eq: op hf} reduces to a trivial (constant) solution in the space of diffeomorphisms $\mathcal{O}$. More precisely, at the energy eigen-state $\psi_0(t, x)=e^{-iE_0t}\sqrt{\rho_0(x)}$ with eigen-energy $E_0$, $\phi(x)=\sqrt{\rho_0(x)}$ solves
\begin{align}
    \label{def: TDSE eigen state}
    E_0\phi (x)=-\frac{1}{2}\Delta \phi( x) + \frac{\delta}{\delta \rho}\mathcal{F}_R(|\phi( x)|^2, x)\cdot \phi(x).
\end{align}
By \eqref{eq: fi variation} and assuming $\phi > 0$, we have
\begin{align*}
    \frac{\delta}{\delta \rho }\mathcal{F}_Q(\rho_0, x) &= -2\Delta \log(\phi^2) -\lvert\nabla\log (\phi^2)\rvert^2\\
    &=-4\frac{\phi\Delta \phi-\lvert \nabla\phi\rvert^2}{\phi^2}-4\lvert\frac{\nabla \phi}{\phi}\rvert^2 \\
    &=-4\frac{\Delta\phi}{\phi},
\end{align*}
Together with \eqref{def: TDSE eigen state}, we have 
\begin{align*}
    \frac{\delta}{\delta \rho }\mathcal{F}(\rho_0, x)&=\frac{\delta}{\delta\rho}\mathcal{F}_R(\rho_0, x)+\frac{1}{8}\frac{\delta}{\delta\rho}\mathcal{F}_Q(\rho_0, x)\\
    &=\frac{E_0\phi(x)+\frac{1}{2}\Delta \phi( x)}{\phi}+\frac{1}{8}(-4\frac{\Delta\phi}{\phi})\\
    &=E_0
\end{align*}
Hence \eqref{def: op el} becomes
\begin{align}
    \ddot{\boldsymbol{T}} 
    =-\nabla_XE_0=0.
\end{align}
This corresponds to a trivial solution in the space of diffeomorphisms (constant speed).

\section{Parameterized generative model of TDSE} \label{sec: parameterized ge}
The generative model for TDSE \eqref{eq: op hf} is formulated in the space of diffeomorphisms $\mathcal{O}$ and its tangent bundle, which are infinite dimensional spaces. To simulate its solution, we must approximate it in a finite dimensional space. To this end, we follow the ideas proposed in the parameterized WHF (PWHF) \cite{wu2025parameterized}. 

Instead of arbitrary $T \in \mathcal{O}$, we restrict the consideration in a subspace of $\mathcal{O}$ in which each $T$ can be parameterized by $\theta$, denoted as $T_{\theta}$. There may be different choices for $T_{\theta}$, such as finite element or Fourier approximations. In this study, we choose neural networks, especially the Neural ODEs described in Section \ref{sec: tools}. In other words, $\mathcal{O}$ is replaced by  $\mathcal{O}_\theta=\{T_{\theta}: ~\theta\in \Theta \subset \mathbb{R}^{m} \}$ in the problem setting in Section \ref{sec: TDSE}.  The tangent space of $\mathcal{O}_\theta$ is $\mathcal{T}_{\theta}\mathcal{O}=\text{span}\{\partial_{\theta_k}T_{\theta}:k=1, \cdots, m\}$; the metric \eqref{def: op metric} becomes
\begin{align}
    G(\theta)=\int  \partial_{\theta} T_\theta(z)^{\top} \partial_{\theta} T_\theta(z)\lambda(z)~dz; \label{relaxed metric tensor}
\end{align}
and the Lagrangian in the parameter space $\Theta$ is reformulated to
\begin{equation}
\label{eq:para-Lagrangian}
    L(\theta, \dot\theta) = \frac{1}{2}\dot\theta^\top G(\theta)   \dot\theta-F(\theta),
\end{equation}
where $F(\theta):=\mathcal{F}(\rho_{\theta})$. Hence the variational formulation \eqref{def: op vari} is reduced to 
\begin{align}
    \label{vari exact PWHF}
    \mathcal{I}^\Theta (\theta)= \inf_{\theta}\cbr[2]{\int_0^{t_0} L(\theta(t), \dot{\theta}(t))dt: \rho_{\theta(0)}=\rho_0, \rho_{\theta(t_0)}=\rho_{t_0} }.
\end{align}

Mimicking the derivation of PWHF as detailed in \cite{wu2025parameterized}, we obtain the parameterized generative model for TDSE \eqref{eq: op hf}.
\begin{proposition} The critical point of $\mathcal{I}^\Theta (\theta)$ defined in \eqref{vari exact PWHF} satisfies a Hamiltonian system 
  \begin{subequations}
\label{general PWGF pseudo simplified}
    \begin{align}
        \dot\theta &= G^{\dagger}p,\\
        \dot p &= \frac{1}{2}[(G^{\dagger}p)^{\top} (\partial_{\theta_k}G)G^{\dagger}p]_{k=1}^{m}-\nabla_{\theta}F(\theta),
    \end{align}
\end{subequations}
where $G^{\dagger}$ is the Penrose-Moore pseudo inverse of $G$.  The corresponding Hamiltonian is 
\begin{align}
    \label{eq:pseudo-Hamiltonian}
    H(\theta, p)=\frac{1}{2}p^{\top}\widehat{G}^{\dagger}(\theta) p + F(\theta).
\end{align}
\end{proposition}

Solving \eqref{general PWGF pseudo simplified} can provide an approximate solution $T_{\theta(t)}$ to the original flow \eqref{def: op el} defined in the space of diffeomorphisms $\mathcal{O}$. This further gives the approximate density $\rho_{\theta} = T_{\theta \sharp}\lambda$ in $\mathcal{P}(M)$. Similarly, we can approximate $\nabla\Phi(t, x)$ by $\partial_{\theta} T_{\theta(t)}\circ T_{\theta(t)}^{-1}(\cdot)p(t)$. Together, they form the foundation for our algorithm to compute the parameterized generative model for TDSE. 

\section{Numerical algorithm}
\label{sec: numerical method}

To simulate the parameterized generative model of TDSE \eqref{general PWGF pseudo simplified}, we suggest the numerical method presented in Aglorithm \ref{alg:HFsolver}, which is based on a semi-implicit symplectic Euler scheme proposed in \cite{wu2025parameterized}. Here we omit its details.

\begin{algorithm}[H]
\caption{Parameterized TDSE solver}
\label{alg:HFsolver}
\begin{algorithmic}\STATE{Initialize neural network $T_{\theta}$ with parameters $\theta^0$ at $t=0$.}
\STATE{Initialize $p^0=\nabla_{\theta}\mathbb{E}_{z\sim\lambda}[\Phi(0, T_{\theta}(z))]$ with $\Phi(0, x)=-i\log \frac{\psi(0, x)}{|\psi(0, x)|}$.}
\STATE{Set terminal time $t_1$, number of steps $K$ and fixed point iteration rate $\gamma$, define the step size as $h=t_1/K$. }
\FOR{$l=0, \cdots, K-1$}
\STATE{Sample $\{z_1, \cdots, z_N\}$ from $\lambda$, and compute $X_i=T_{\theta^l}(z_i)$ as well as $\rho_{\theta^l}(x_i)=T_{\theta^l\sharp} \lambda(x_i)$}
\STATE{Apply MINRES to solve $\xi^{l, 0}$ from equation $G(\theta^{l})\xi =p^l$, set $\alpha^{l, 0}=\theta^l$}
\FOR{$j=1, \cdots, n_{in}$}
\STATE{Update $\alpha^{l, j}= \theta^l+h\xi^{l, j}$}
\STATE{Update $\xi^{l,j+1} = \xi^{l,j} - \gamma (G(\alpha^{l,j}) \xi^{l,j} - p^l)$
}
\ENDFOR
\STATE{Set $\theta^{l+1}=\alpha^{l, n_{in}}, \eta^{l+1}=\xi^{l, n_{in}}$}
\STATE{Compute $X_i^{l+1}=T_{\theta^{l+1}}(z_i)$ as samples from $\rho_{\theta^{l+1}}$, evaluate $\nabla_{\theta}F(\theta^{l+1})$ through samples $\{X_i^{l+1}\}$}  
\STATE{Set $p^{l+1}=p^l+\frac{h}{2}[(\eta^{l+1})^{\top} \partial_{\theta_k}G\eta^{l+1}]_{k=1}^{m}-h\nabla_{\theta}F(\theta^{l+1})$}
\ENDFOR
\STATE{\textbf{Output:} Solution at discrete time spots: $(\theta(lh) := \theta^{l}, p(lh) := p^l)$ for $l=0,\dots,K$. 
}
\end{algorithmic}
\end{algorithm}

Algorithm \ref{alg:HFsolver} has several important features. First, the algorithm does not require spatial discretization or basis, and is sampling-based and can be executed as long as samples $\{z_1, \cdots, z_{N}\}$ from the reference distribution are available. 

This is a particularly important and the main reason that Algorithm \ref{alg:HFsolver} can scale up to high-dimensional cases. 
In practice, samples from the reference distribution can be taken from the same or different sets at each time step. 
Second, the parameters $\theta$ are computed by iterative methods for linear least squares problems. This is different from typical network training where the parameters are solved from challenging large-scale non-convex optimization problems. In particular, Algorithm \ref{alg:HFsolver} can be implemented without explicitly creating the matrix $G$, because it only requires matrix vector multiplication. The minimal residual (MINRES) method is chosen for convenience. It can be replaced by other numerical solvers as long as the least-squares solution for the linear system can be obtained efficiently. 
Third, compared to the algorithm proposed in \cite{wu2025parameterized}, we develop a new approach to evaluate the Fisher information appeared in $F(\theta)$ by leveraging the properties of the Neural ODE $T_{\theta}$. A detailed PyTorch-based algorithm is given in Appendix \ref{sec: evaluation of fi}. 

There are many options to initialize $\theta^0$ in Algorithm \ref{alg:HFsolver} depending on the problem settings. 
We can select $T_{\theta^0}$ the identity map if $\rho_0$ is the same as the reference density, or samples at $t=0$ are given. In the latter case, the algorithm can be carried out without the reference distribution. %
In general, we can initialize $\theta^0$ by minimizing the difference between $\rho_{\theta^0}$ and $\rho_0$, for example,    
$$
\theta^0=\underset{\theta}{\textrm{argmin}} \{\mathcal{D}_{\mathrm{KL}}(\rho_0 \| \rho_{\theta})\},
$$
where $\mathcal{D}_{\mathrm{KL}}$ stands for the Kullback-Leibler divergence between two probability densities. Other distance or divergence can be used as well. 
Once $\rho^0$ is available, $p^0$ can be initialized by $\nabla_{\theta}\mathbb{E}_{z\sim\lambda}[\Phi(0, T_{\theta}(z))]$, where the initial phase $\Phi(0,x)$ is obtained through the Madelung transform.

\section{Numerical results}
\label{sec:example}
In this section, we present three examples to demonstrate the performance of the proposed formulation and algorithm. The first is the quantum harmonic oscillator; the second is the Gross-Pitaevskii equation (GPE); and the third is a three-particle system. We focus on their numerical simulations. Our computation is carried out on a desktop computer with an NVIDIA RTX-4080s GPU (16 GB memory) and CUDA enabled.  

\subsection{Coherent state solution to the quantum harmonic oscillator} 
Although the appearance of the parameterized system \eqref{general PWGF pseudo simplified} is different from the TDSE given in \eqref{def: TDSE}, it can be verified that they are theoretically equivalent in some simplified situations. For example, let us consider the linear Schr\"odinger equation with quadratic potential
\begin{align}
\label{eq: quantum ho}
    &i\partial_t\psi =-\frac{1}{2} \Delta \psi + \frac{x^2}{2}\psi,\ x=(x_1, x_2)\in \mathbb{R}^2,\\
    &\psi(0, x) = \frac{1}{\pi^{1/4}}\exp\left(-\frac{1}{2}\left[(x_1-\sqrt{2})^2 + x_2^2\right]+i\sqrt{2} x_2\right).
\end{align}
This equation \eqref{eq: quantum ho} is known as the quantum harmonic oscillator whose coherent state solution can be verified analytically as 
\begin{align}
\label{eq: coherent state1}
    \psi(t, x)=\frac{1}{\pi^{1/4}}\exp\left(-\frac{1}{2}\lvert x-\sqrt{2}\alpha (t))\rvert^2+i\sqrt{2}\beta(t)^\top x+ic(t)\right),
\end{align}
where $\alpha(t)=(\cos t, \sin t)^\top, \beta(t)=(-\sin t, \cos t)^\top$, and $c(t)$ is a time dependent function satisfying $c(0)=0$.

If we select an affine transformation as the push-forward map, i.e., 
\begin{align}
    \label{def: affine map}
    T_{\theta}(z)=A(t)z+b(t), \quad \text{with parameters} \quad \theta(t)=(A(t), b(t)), \ A(t)\in\mathbb{R}^{d\times d}, \ b(t)\in\mathbb{R}^d,
\end{align}
and set the reference density $\lambda$ as
\begin{align}
    \lambda(z)=\frac{1}{\sqrt{\pi}}\exp\left(-\lvert z\rvert^2\right), 
\end{align}
the system for the parameters \eqref{general PWGF pseudo simplified} becomes 
\begin{subequations}
\label{eq: coherent state para ode1}\begin{align}
    &\ddot A(t)=-\nabla_A\left(\frac{1}{4}\textrm{Tr}\left[ A(t)^\top A(t)+A(t)^{-\top} A(t)^{-1}\right]\right), \quad \ddot b(t)=-\nabla_{b}\left(\frac{1}{2}b(t)^\top b(t)\right),\\
    &A(0)=I_{d\times d},\ \dot A(0)=0_{d\times d}, \quad b(0)=(\sqrt{2}, 0)^\top, \  \dot b(0)=(0, \sqrt{2})^\top.
\end{align}
\end{subequations}

Its solution can be verified as
\begin{align}
\label{eq: para qho sol1}
    A(t)=I_{d\times d}, \quad b(t)=\sqrt{2}\alpha(t).
 \end{align}
Plugging \eqref{eq: para qho sol1} into \eqref{eq: change of vari}, we get the parameterized density function 
\begin{align}
    \rho_{\theta}(x)=\frac{1}{\sqrt{\pi}}\exp\left(-\lvert x-\sqrt{2}\alpha (t)\rvert^2\right),
\end{align}
which equals to the density function given by the true solution \eqref{eq: coherent state1}. 

We use Algorithm \ref{alg:HFsolver} to solve the parameter ODEs \eqref{eq: coherent state para ode1} for $d=2$. A total of $N=20,000$ samples are used in the calculation. The matrix $A(t)$ remains as an identity and the vector $b(t)-b(0)$ is depicted in Figure \ref{fig: QHO}. Clearly, the numerical solution coincides with the exact solution as our analysis suggested in this example.

\begin{figure}[H]
    \centering
    \includegraphics[width=0.5\linewidth]{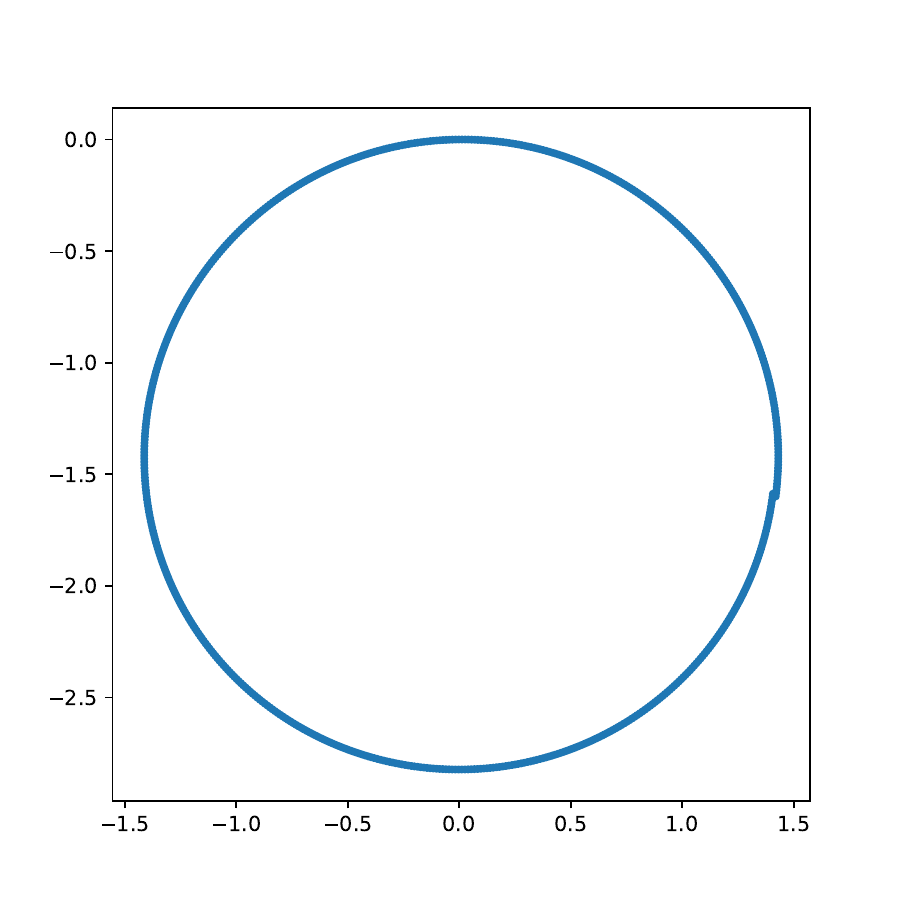}
    \caption{The blue curve shows the trajectory of $b(t)-b(0)$ for the harmonic oscillator with quadratic potential example \eqref{eq: coherent state para ode1}.} 
    \label{fig: QHO}
\end{figure}

\subsection{Gross-Pitaevskii equation}\label{sec: GPE} The Gross-Pitaevskii equation (GPE) is 
\begin{align} \label{def: gpe}
        i\frac{\partial}{\partial t}\psi (t, x)=-\frac{1}{2}\Delta \psi(t, x) + V(x)\psi(t, x)+\gamma |\psi(t, x)|^2\psi(t, x),
\end{align}
where $\gamma\in \mathbb{R}$ is a constant. The GPE is a special case of the TDSE \eqref{def: TDSE} if taking $\mathcal{F}_R(\rho)=\int V(x)~\rho(x)dx+\frac{\gamma}{2}\int \rho(x)^2dx$. When $V(x)\equiv0, \ d=1, \ \gamma=-2$, the GPE \eqref{def: gpe} admits a set of solutions with explicit expression given by 
\begin{align}
    \label{eq: gpe explicit sol}
    \psi(t, x)=\sqrt{\mu/2}~\exp\left(i\left[\frac{1}{2}v x+\frac{1}{2}(\mu -\frac{1}{4}\lvert v\rvert^2)t\right]\right)\textrm{sech}\left(\sqrt{\mu}(x-\frac{1}{2}vt)\right),
\end{align}
where $v\in \mathbb{R}, \mu \in \mathbb{R}^+$ are constants and $\textrm{sech}(y)=\frac{2e^y}{e^{2y}+1}$.

Now we apply Algorithm \ref{alg:HFsolver} to the GPE \eqref{def: gpe} whose exact solution is given by \eqref{eq: gpe explicit sol}. We use Neural-ODE architecture to parameterize the push-forward map. The right-hand side of the Neural-ODE consists of two hidden layers, each containing 50 neurons and the hyperbolic tangent as the activation function. We generate $10,000$ samples for the computation of $G$ metric and potential energy $\mathcal{F}$. The time step size is taken as $h=0.005$ 
unless otherwise stated. We compute the solutions for $\theta(t)$ and use the push-forward $T_{\theta(t)}$ to generate samples for $\rho_{\theta(t)}$. The initial condition is taken from \eqref{eq: gpe explicit sol} with $t=0$. In Figure \ref{1d GPE sampleplot}, we compare the analytical solution for $\rho(t)$ (blue curve) with the histogram generated by $T_{\theta(t)}$ (orange bars) at different time snapshots. The computed solution matches the analytic solution well. 
\begin{figure*}[t!]
    \begin{subfigure}{0.32\textwidth}
        \centering
        \includegraphics[width=0.95\linewidth]{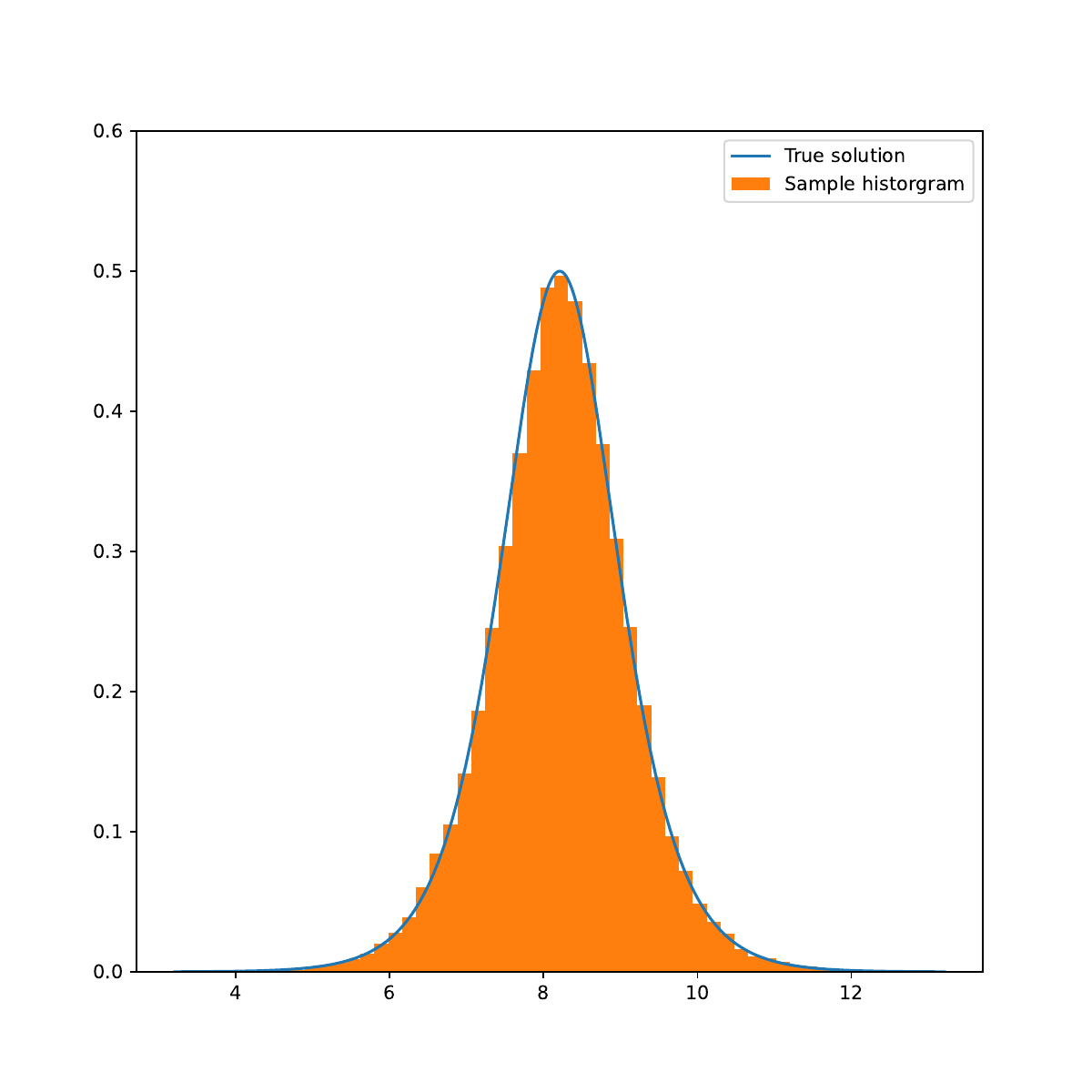}
        \caption{$t=0$}
    \end{subfigure}%
    ~
    \begin{subfigure}{0.32\textwidth}
        \centering
        \includegraphics[width=0.95\linewidth]{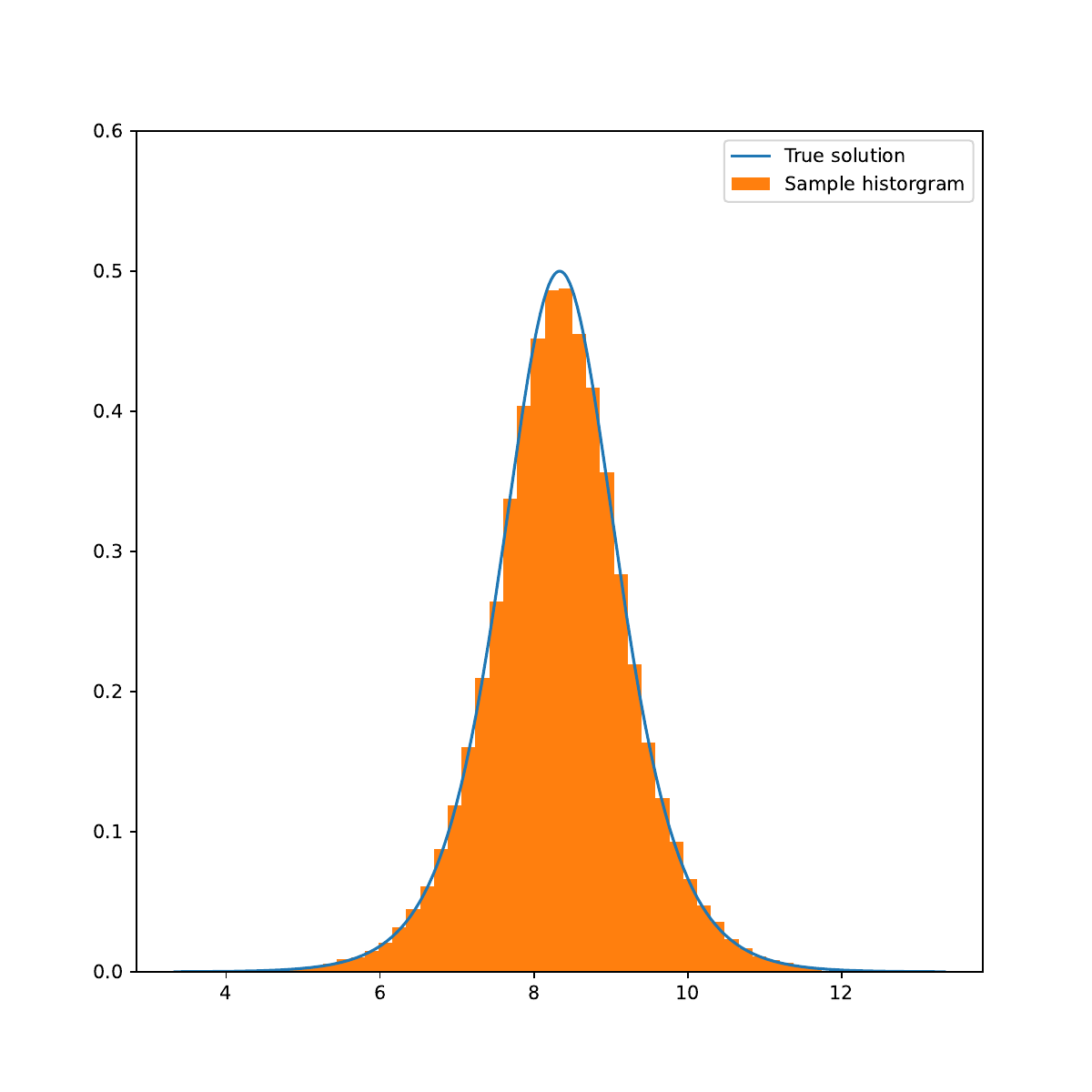}
        \caption{$t=0.25$}
    \end{subfigure}
    ~
    \begin{subfigure}{0.32\textwidth}
        \centering
        \includegraphics[width=0.95\linewidth]{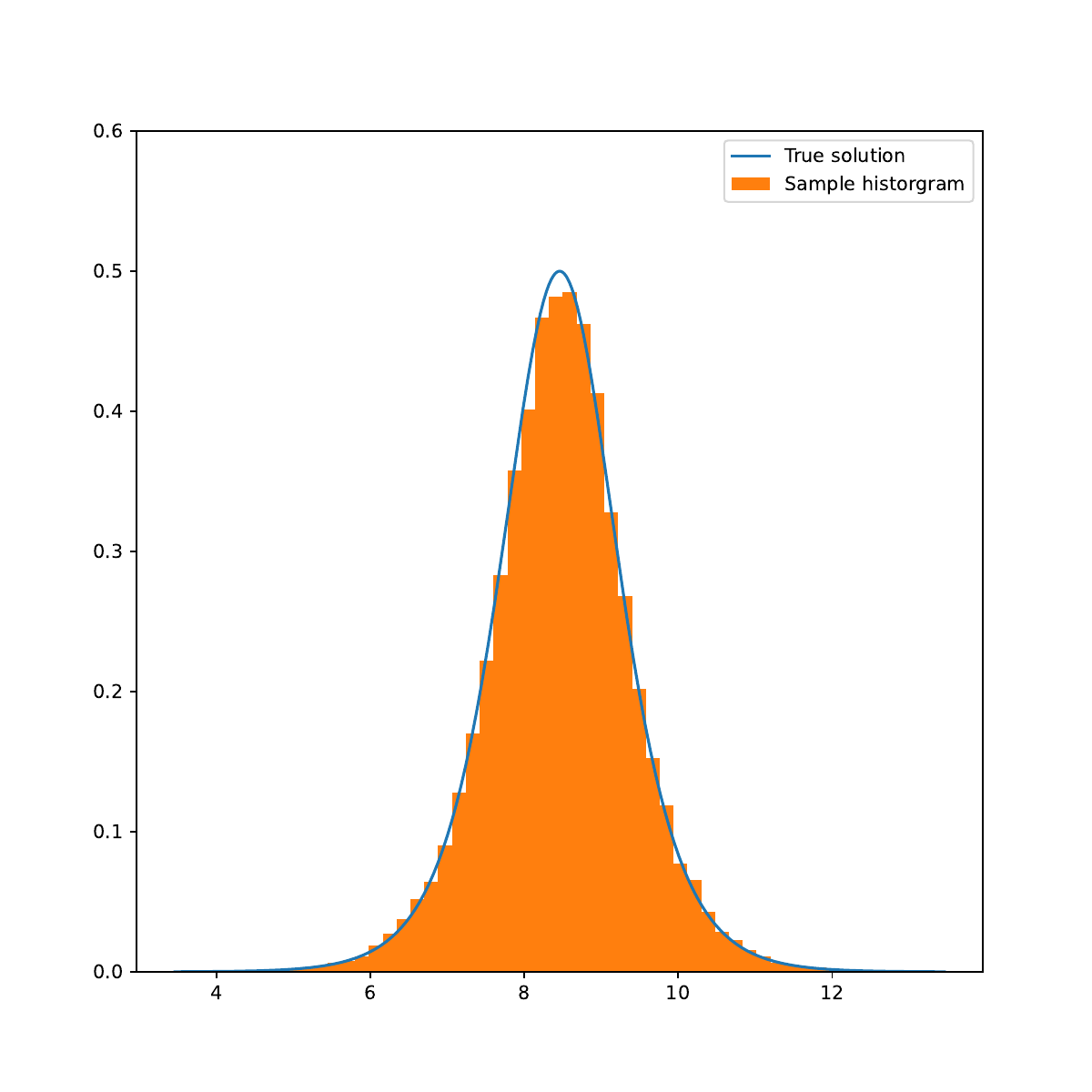}
        \caption{$t=0.5$}
    \end{subfigure}
    \\
    \begin{subfigure}{0.32\textwidth}
        \centering
        \includegraphics[width=0.95\linewidth]{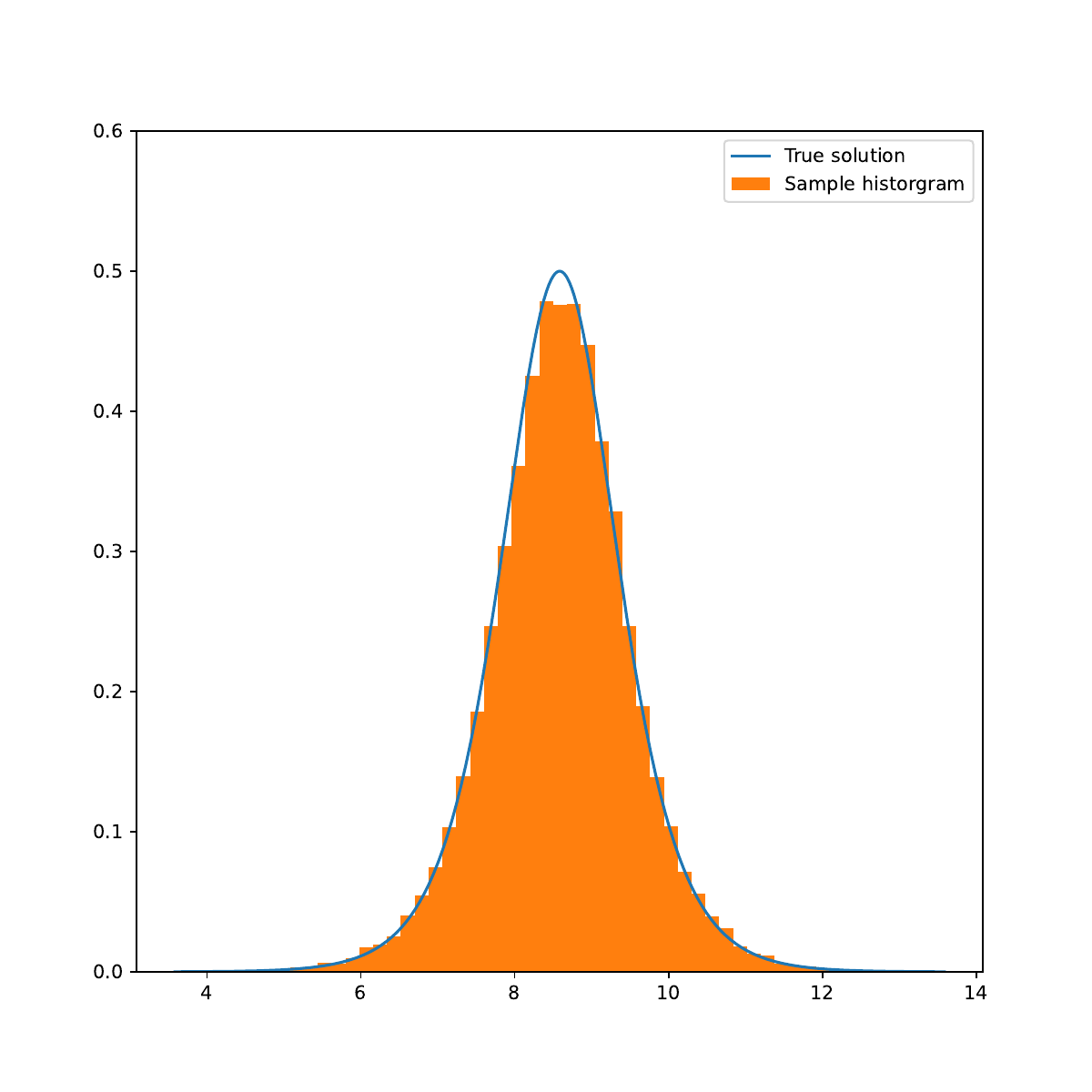}
        \caption{$t=0.75$}
    \end{subfigure}%
    ~
    \begin{subfigure}{0.32\textwidth}
        \centering
        \includegraphics[width=0.95\linewidth]{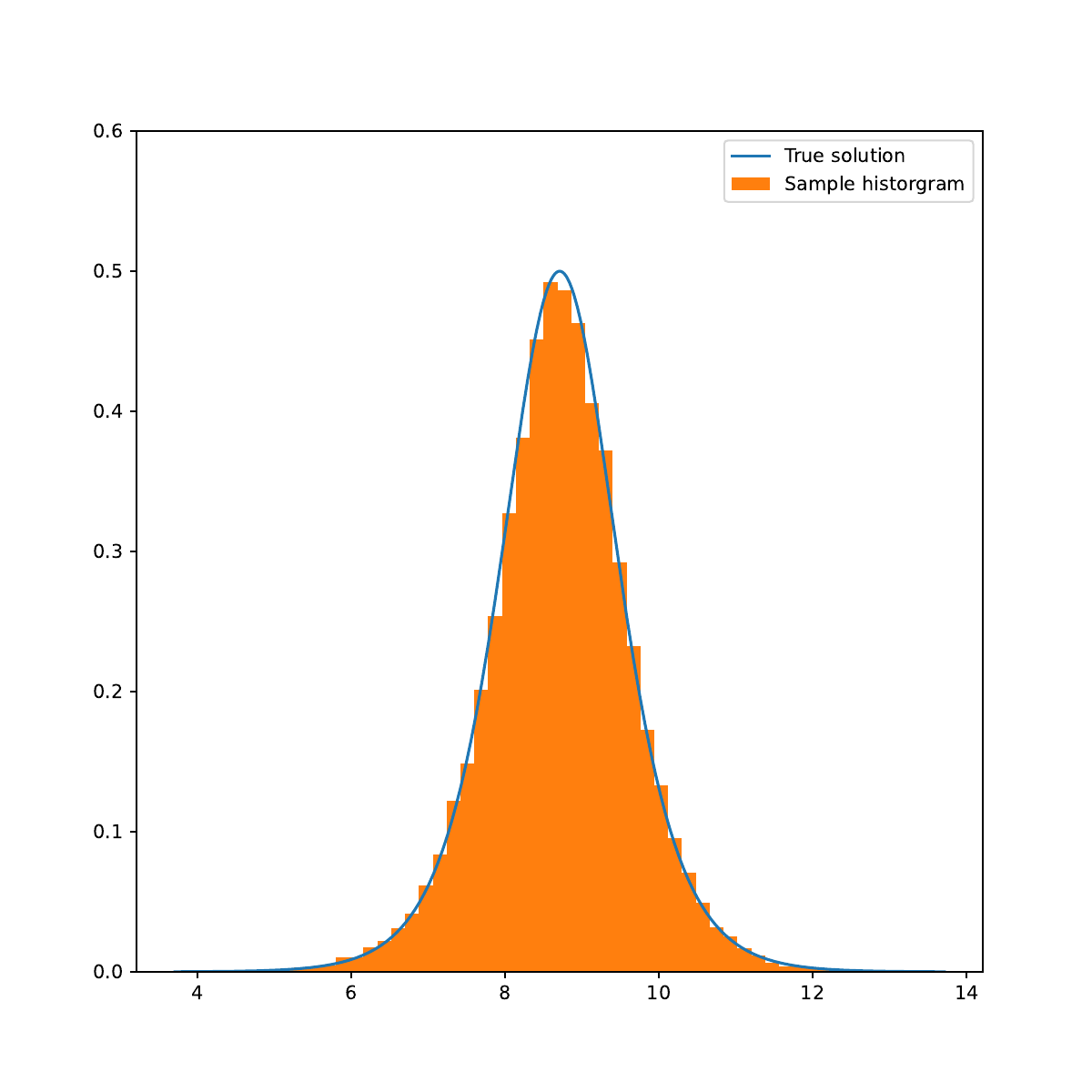}
        \caption{$t=1$}
    \end{subfigure}
    ~
    \begin{subfigure}{0.32\textwidth}
        \centering
        \includegraphics[width=0.95\linewidth]{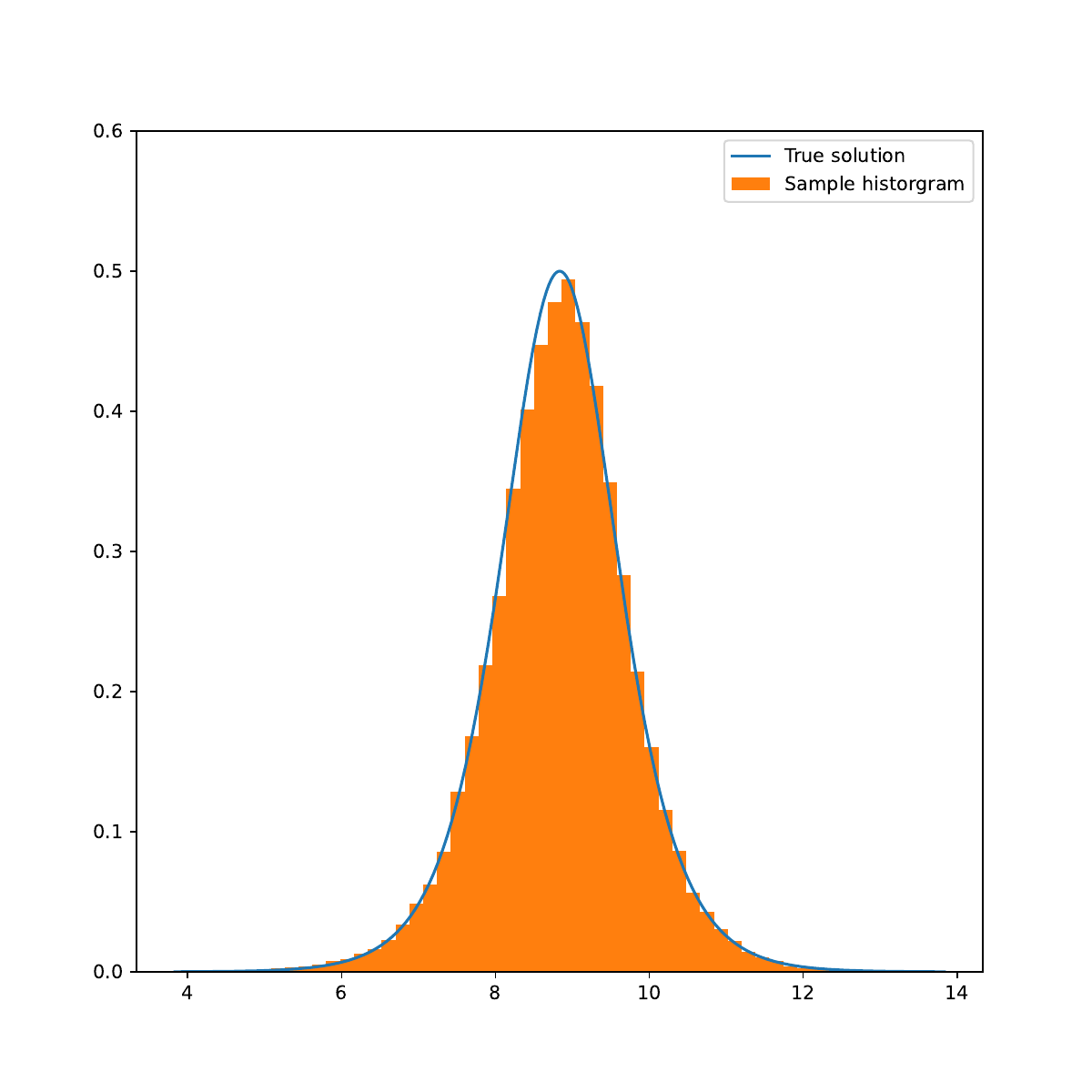}
        \caption{$t=1.25$}
    \end{subfigure}
    \caption{Sample histogram of computed $\rho_{\theta}$ at different time $t$ for $1$D ($d=1$) GPE.}
    \label{1d GPE sampleplot}
\end{figure*}

In the next two experiments, we apply Algorithm \ref{alg:HFsolver} to solve the GPE in $d=3$ and $d=6$, respectively. We also set $V(x)\equiv0, \gamma=-2$ in both experiments. The initial values are taken as the same as $d=1$ with dimension adjustments.  
Figure \ref{3d GPE sampleplot} shows the histograms for the first components of the samples generated by the push-forward map $T_{\theta}$ in the $d=3$ case. Similarly, the histograms are plotted in Figure \ref{6d GPE sampleplot} for $d=6$. We also plot the center of distribution as a function of time in Figure \ref{fig: center of sample 3d} for the $d=3$ experiment.
\begin{figure}[H]
    \centering
    \includegraphics[width=0.5\linewidth]{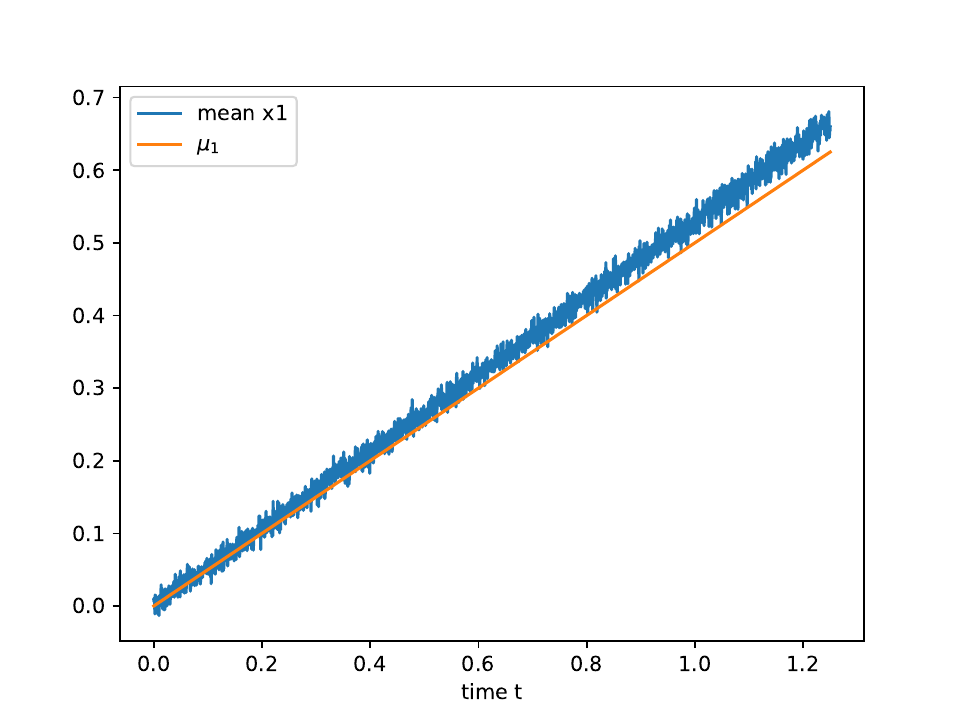}
    \caption{Center of samples for the GPE with $d=3$. The orange curve is the analytical values, the blue curve are obtained by samples generated by the push-forward $T_{\theta(t)}$. 
    }
    \label{fig: center of sample 3d}
\end{figure}

\begin{figure*}[t!]
    \begin{subfigure}{0.32\textwidth}
        \centering
        \includegraphics[width=0.95\linewidth]{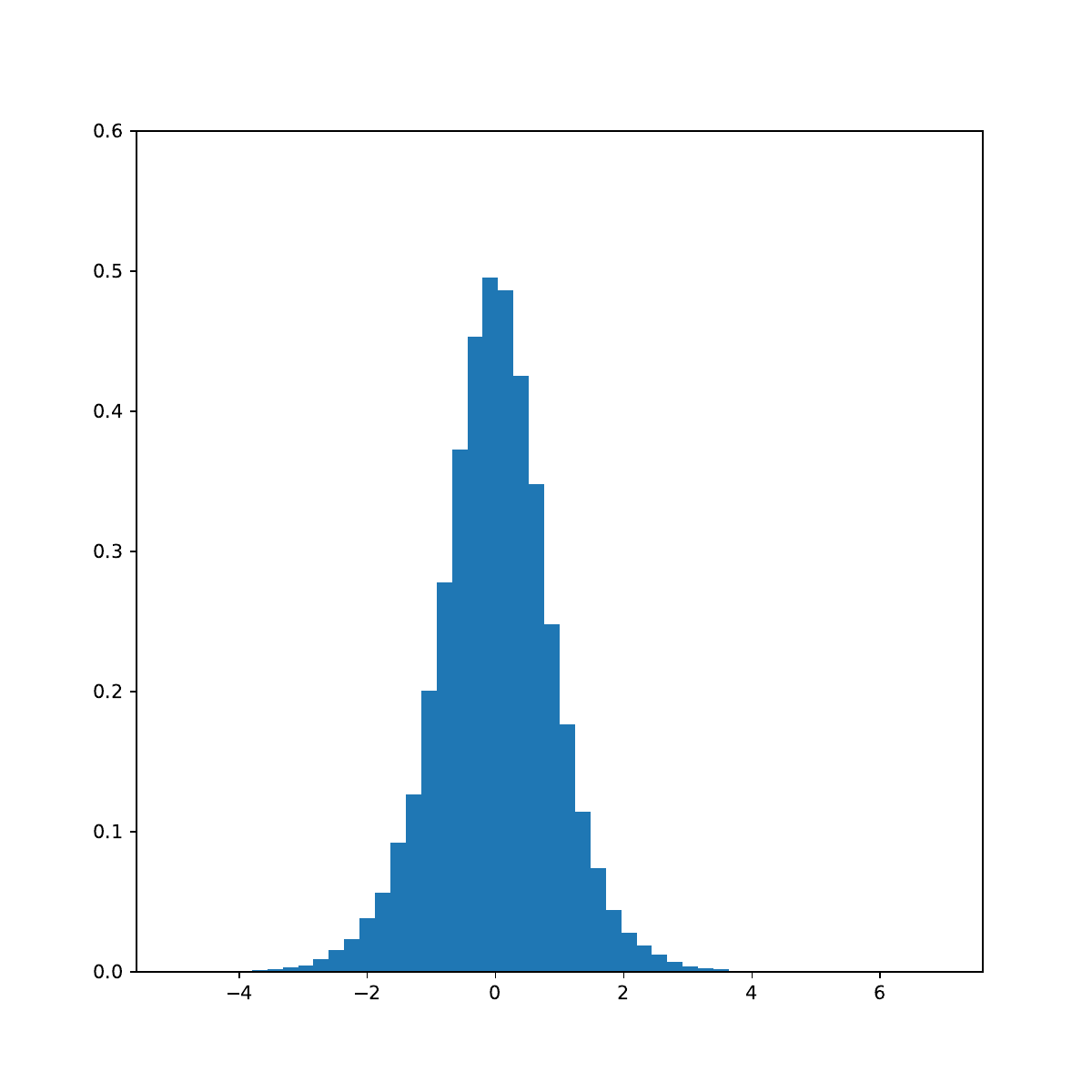}
        \caption{$t=0$}
    \end{subfigure}%
    ~
    \begin{subfigure}{0.32\textwidth}
        \centering
        \includegraphics[width=0.95\linewidth]{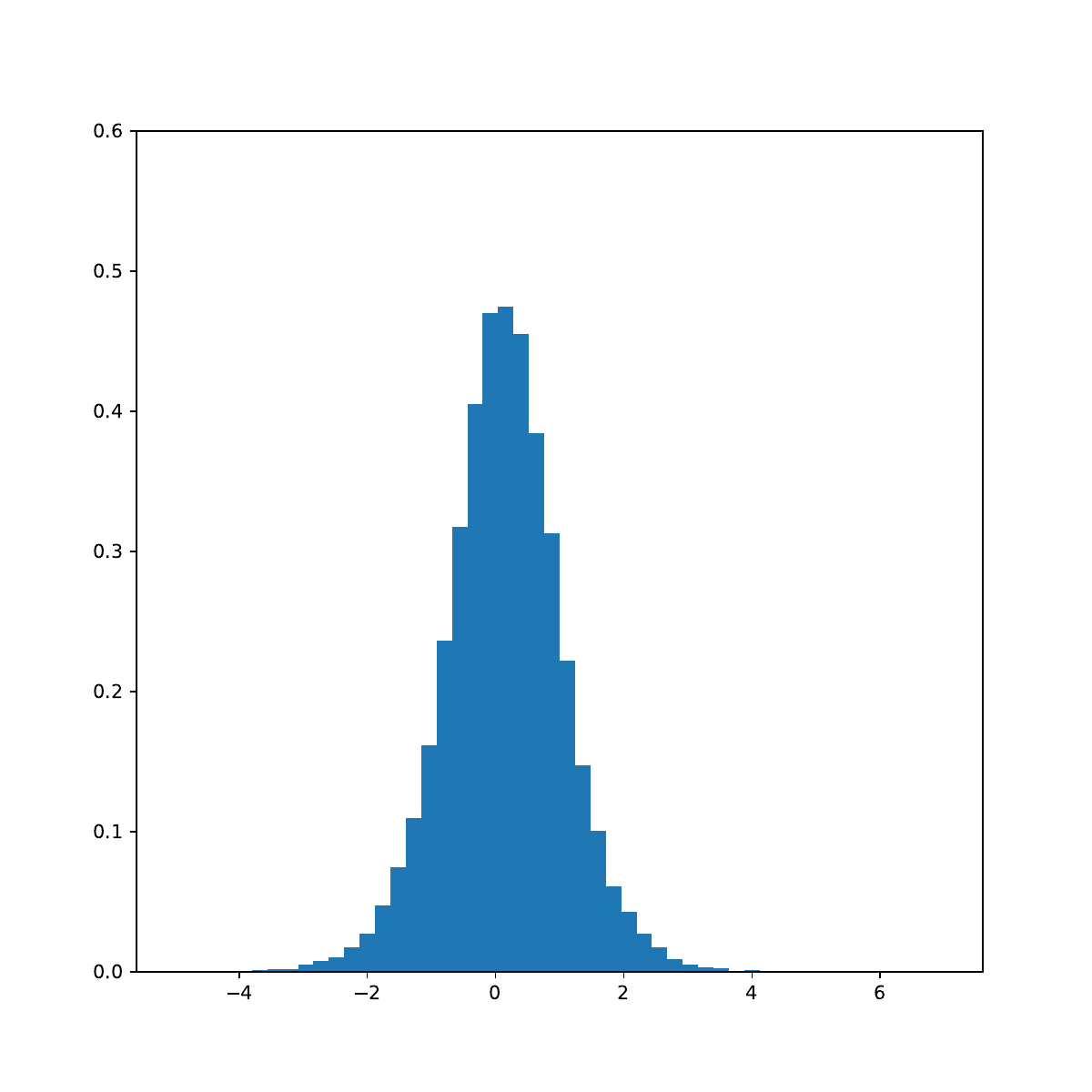}
        \caption{$t=0.25$}
    \end{subfigure}
    ~
    \begin{subfigure}{0.32\textwidth}
        \centering
        \includegraphics[width=0.95\linewidth]{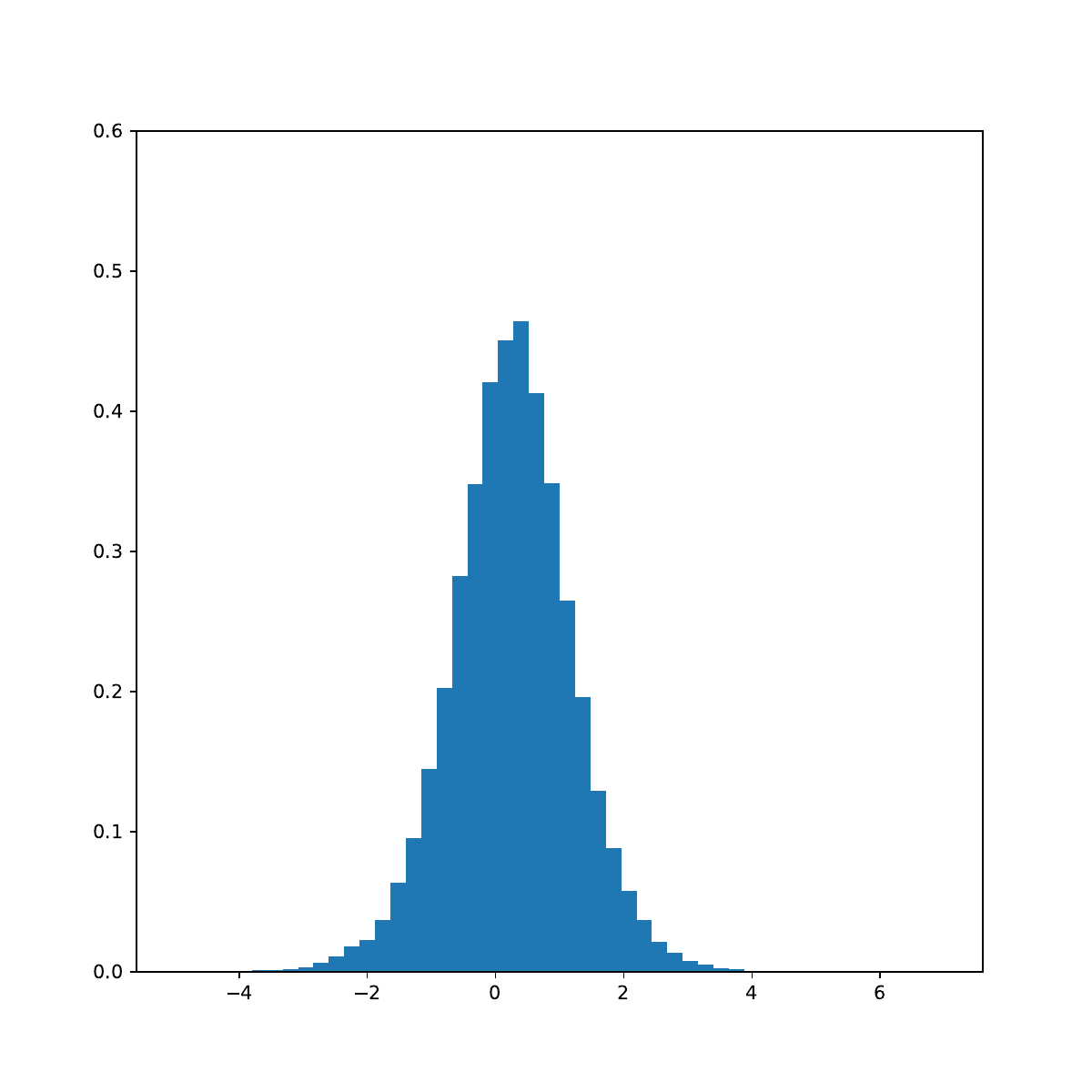}
        \caption{$t=0.5$}
    \end{subfigure}
    \\
    \begin{subfigure}{0.32\textwidth}
        \centering
        \includegraphics[width=0.95\linewidth]{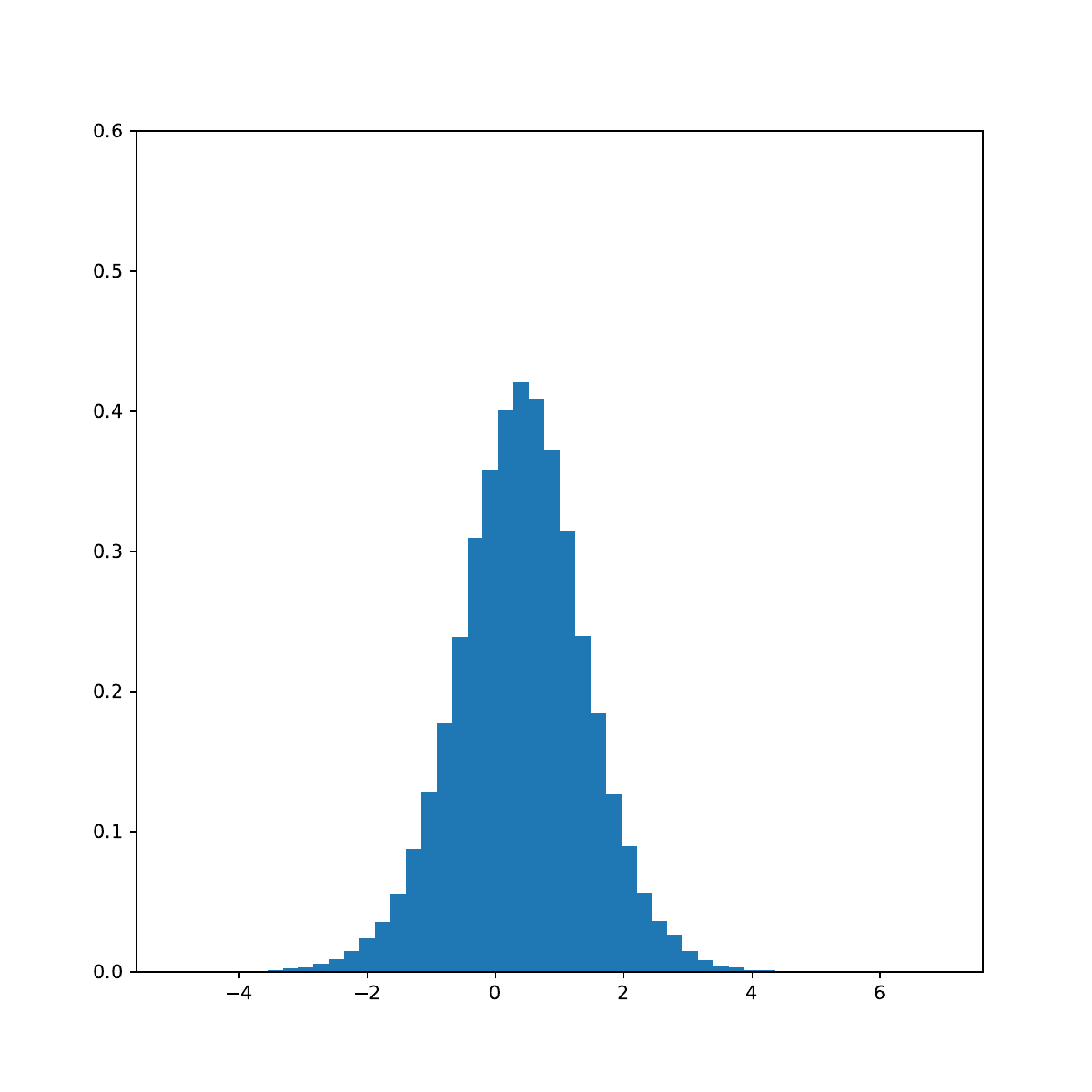}
        \caption{$t=0.75$}
    \end{subfigure}%
    ~
    \begin{subfigure}{0.32\textwidth}
        \centering
        \includegraphics[width=0.95\linewidth]{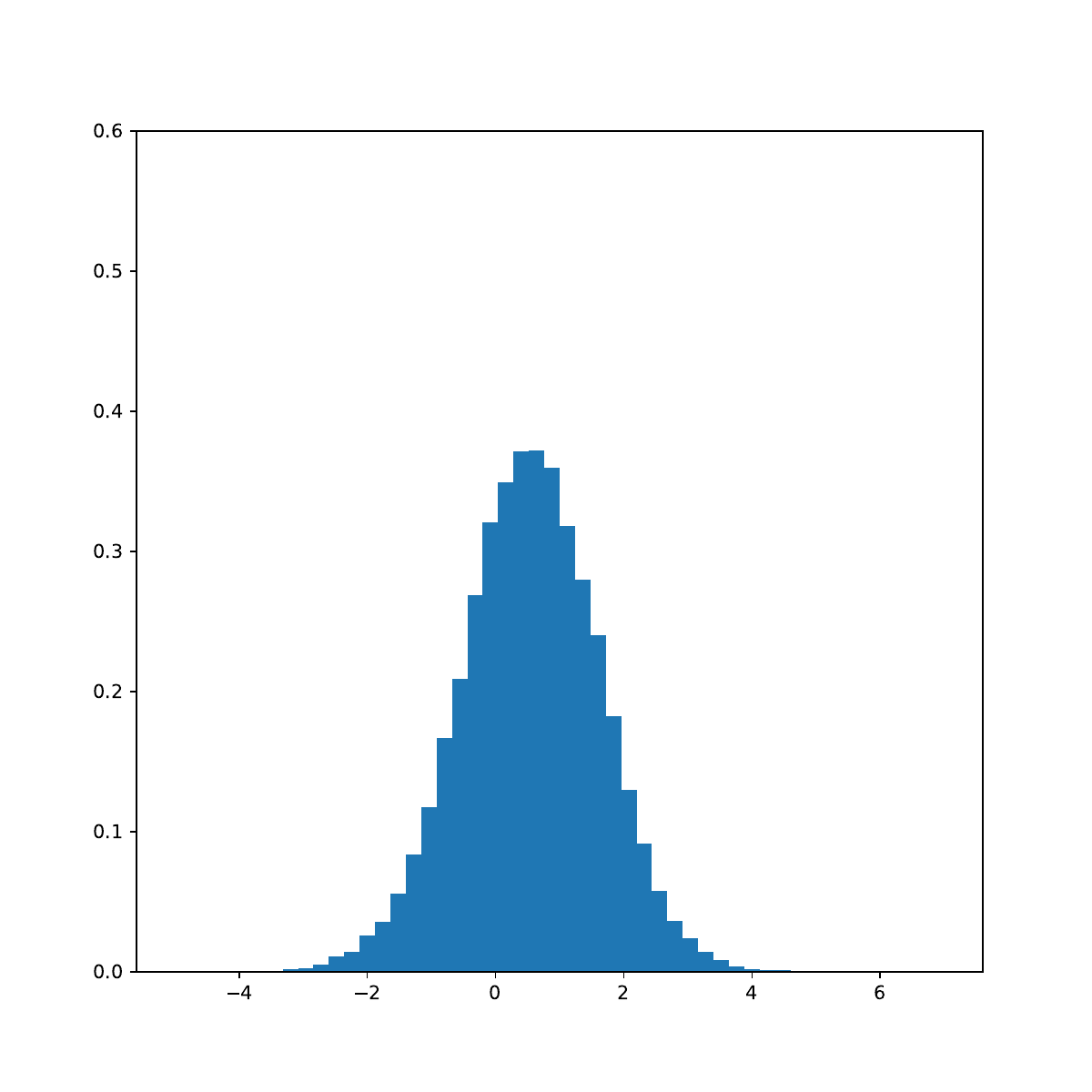}
        \caption{$t=1$}
    \end{subfigure}
    ~
    \begin{subfigure}{0.32\textwidth}
        \centering
        \includegraphics[width=0.95\linewidth]{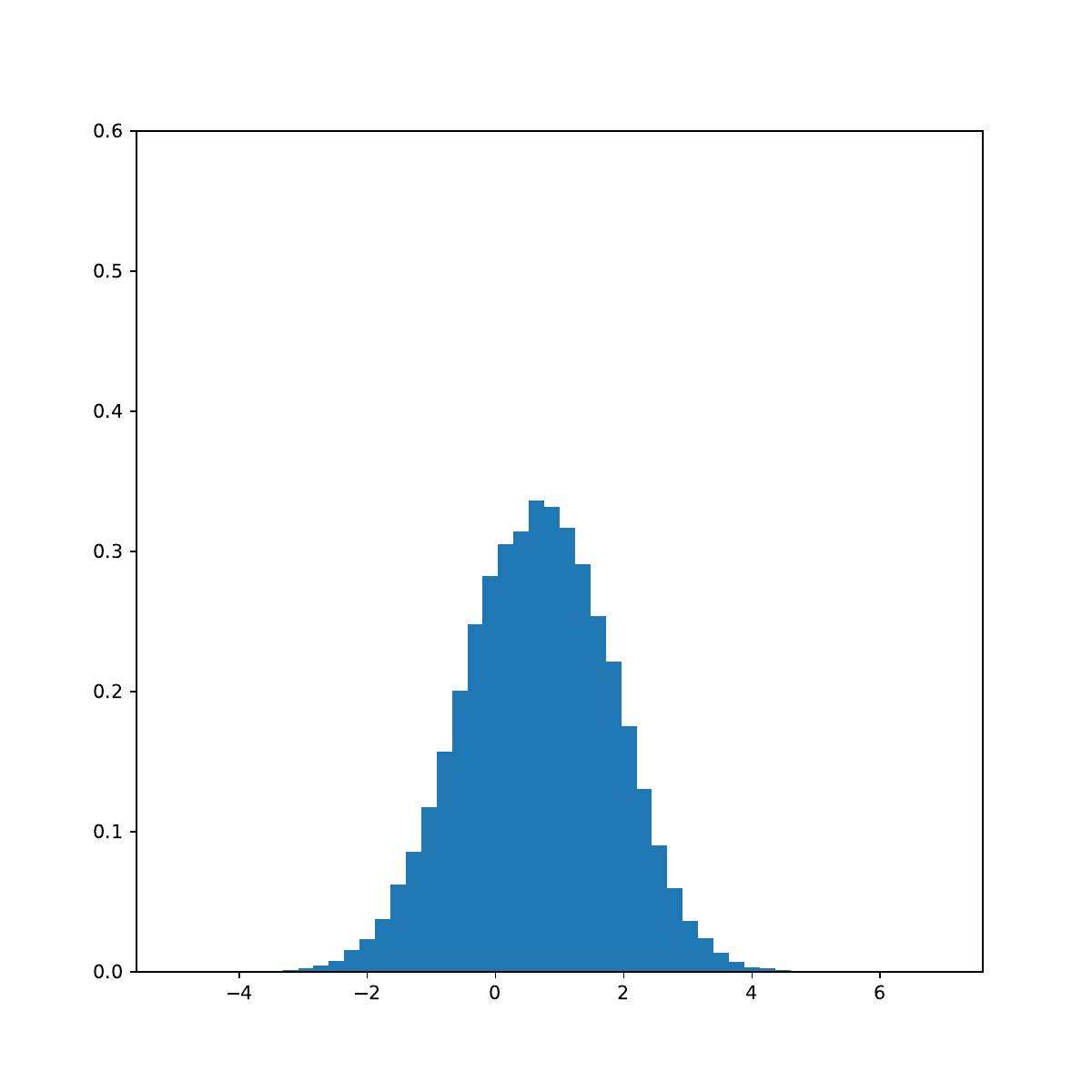}
        \caption{$t=1.25$}
    \end{subfigure}
    \caption{Sample histogram of computed $\rho_{\theta}$ at different time $t$ for $3$D GPE.}
    \label{3d GPE sampleplot}
\end{figure*}

\begin{figure*}[t!]
    \begin{subfigure}{0.32\textwidth}
        \centering
        \includegraphics[width=0.95\linewidth]{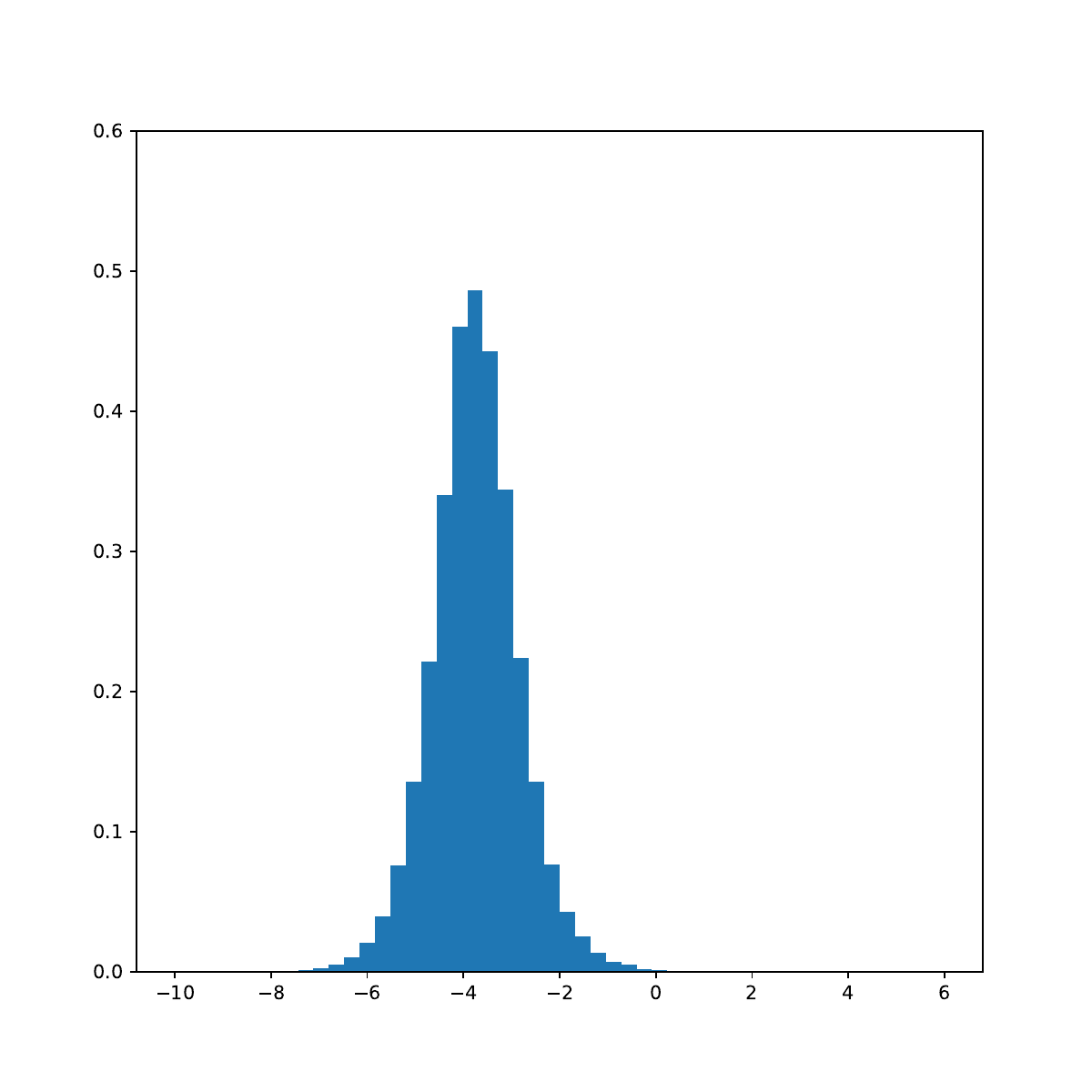}
        \caption{$t=0$}
    \end{subfigure}%
    ~
    \begin{subfigure}{0.32\textwidth}
        \centering
        \includegraphics[width=0.95\linewidth]{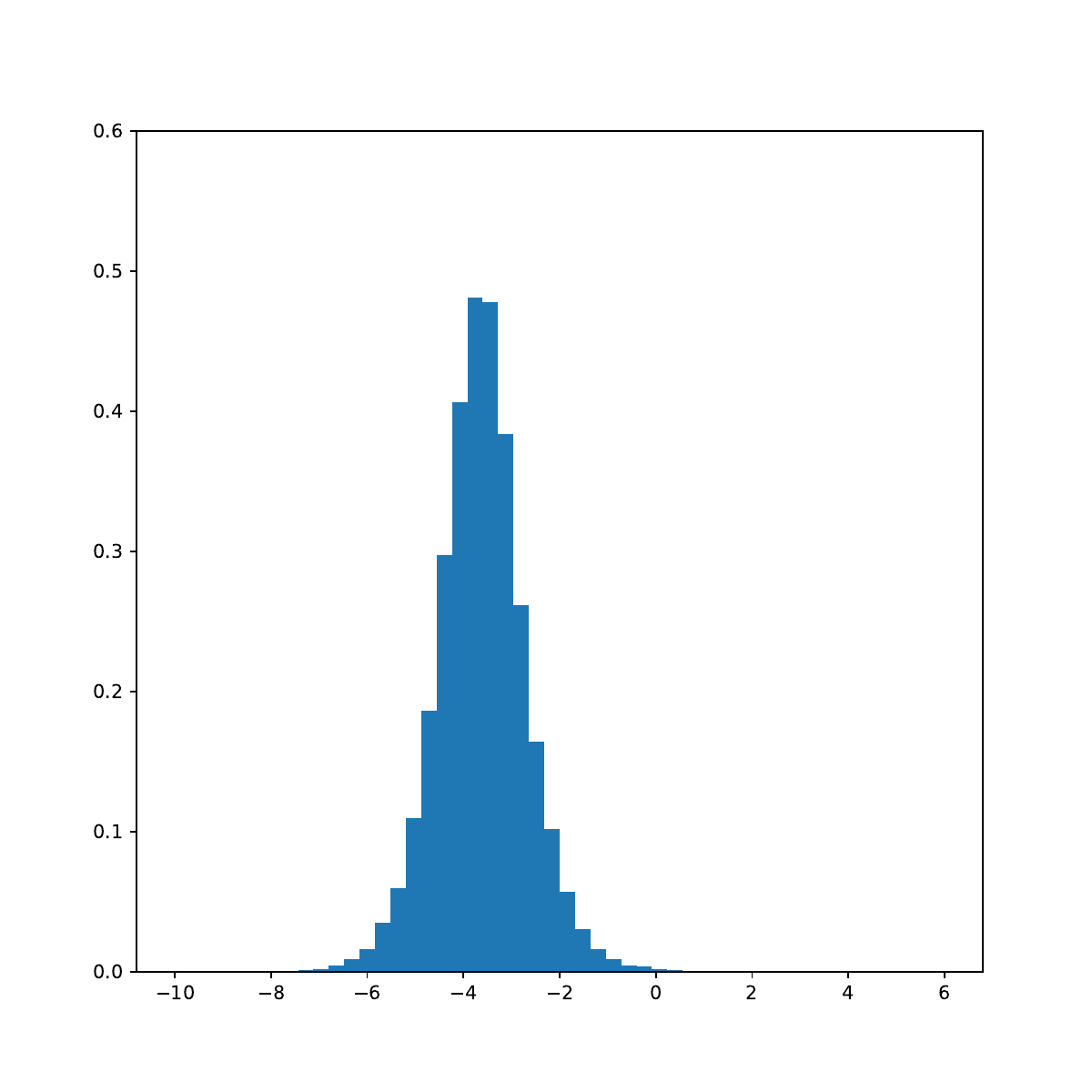}
        \caption{$t=0.25$}
    \end{subfigure}
    ~
    \begin{subfigure}{0.32\textwidth}
        \centering
        \includegraphics[width=0.95\linewidth]{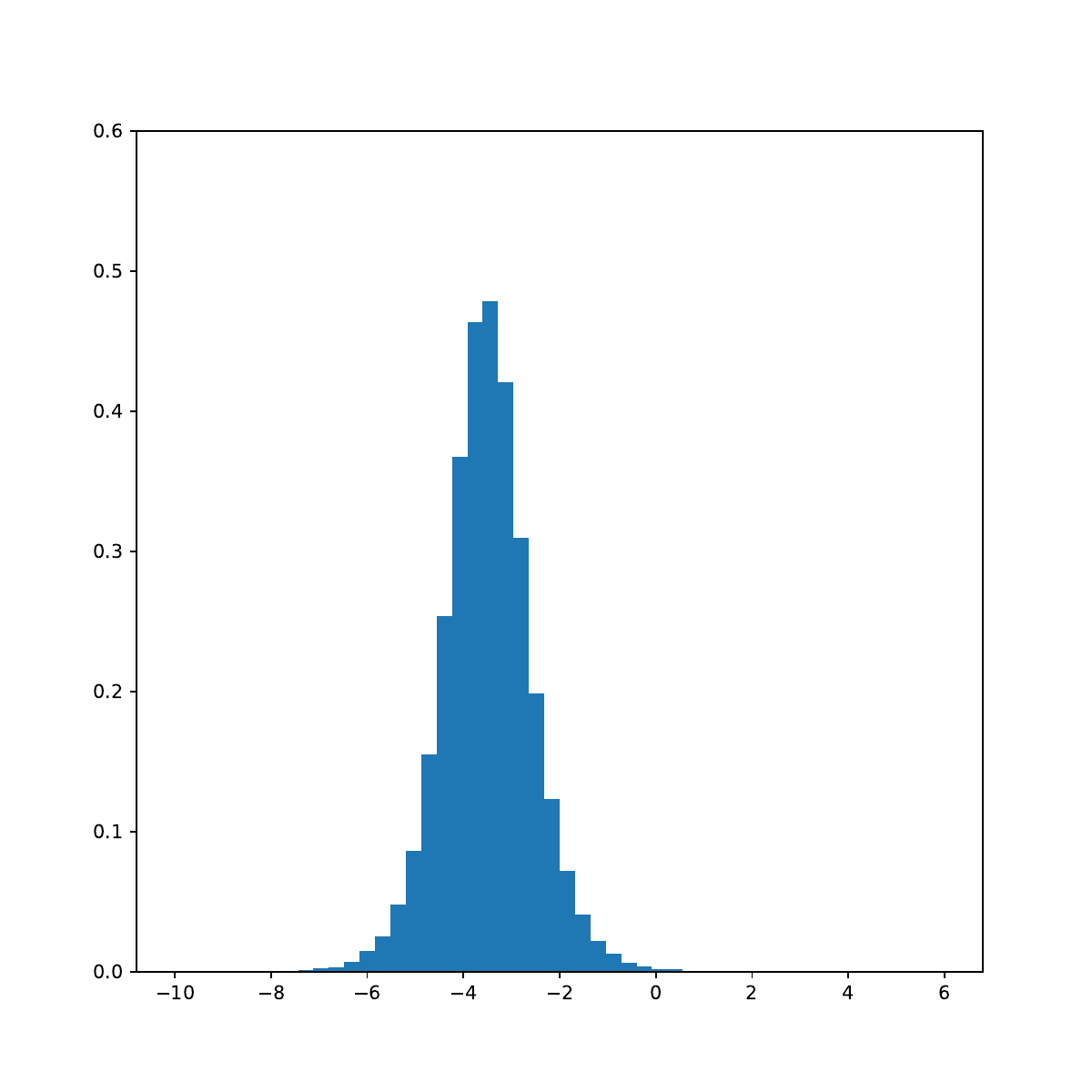}
        \caption{$t=0.5$}
    \end{subfigure}
    \\
    \begin{subfigure}{0.32\textwidth}
        \centering
        \includegraphics[width=0.95\linewidth]{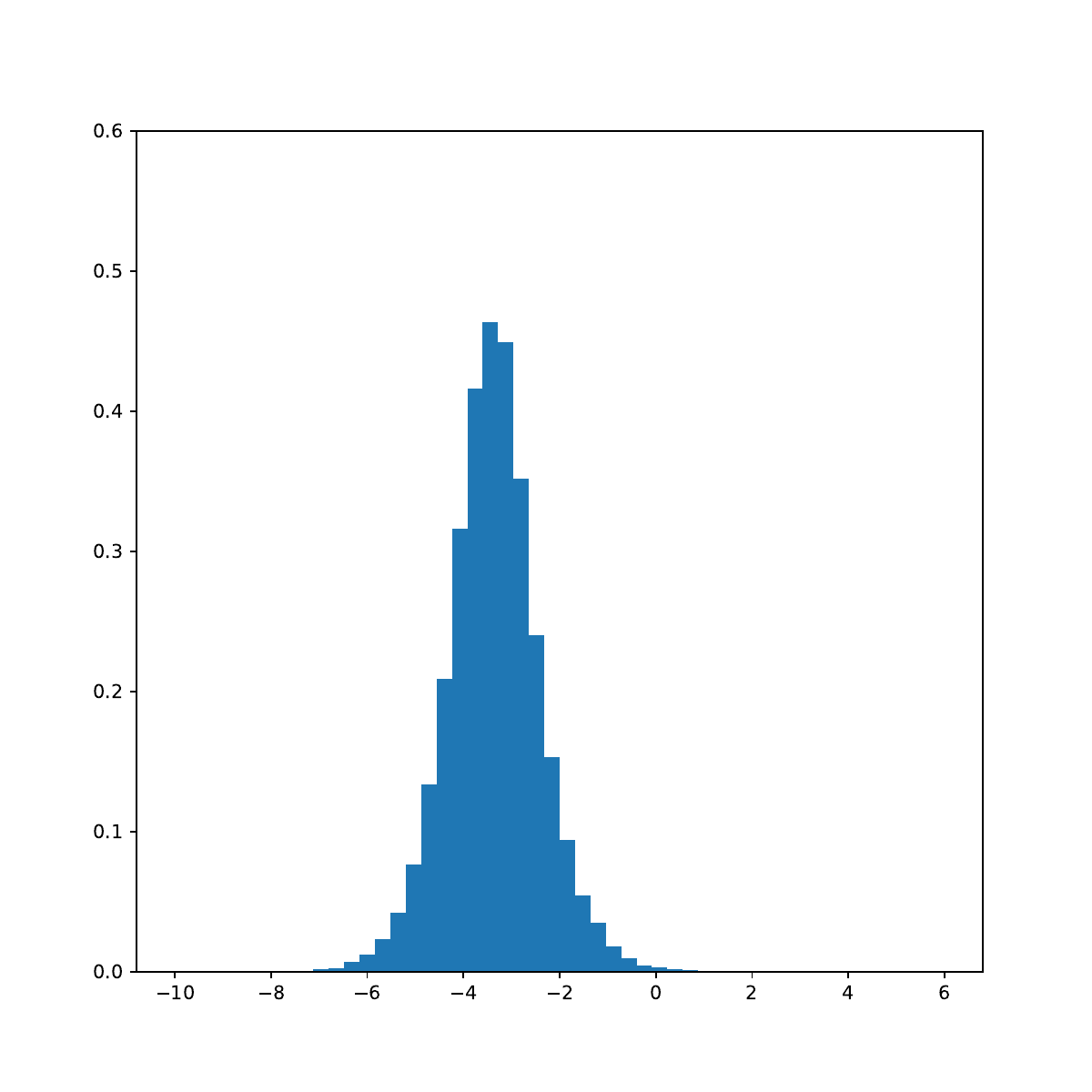}
        \caption{$t=0.75$}
    \end{subfigure}%
    ~
    \begin{subfigure}{0.32\textwidth}
        \centering
        \includegraphics[width=0.95\linewidth]{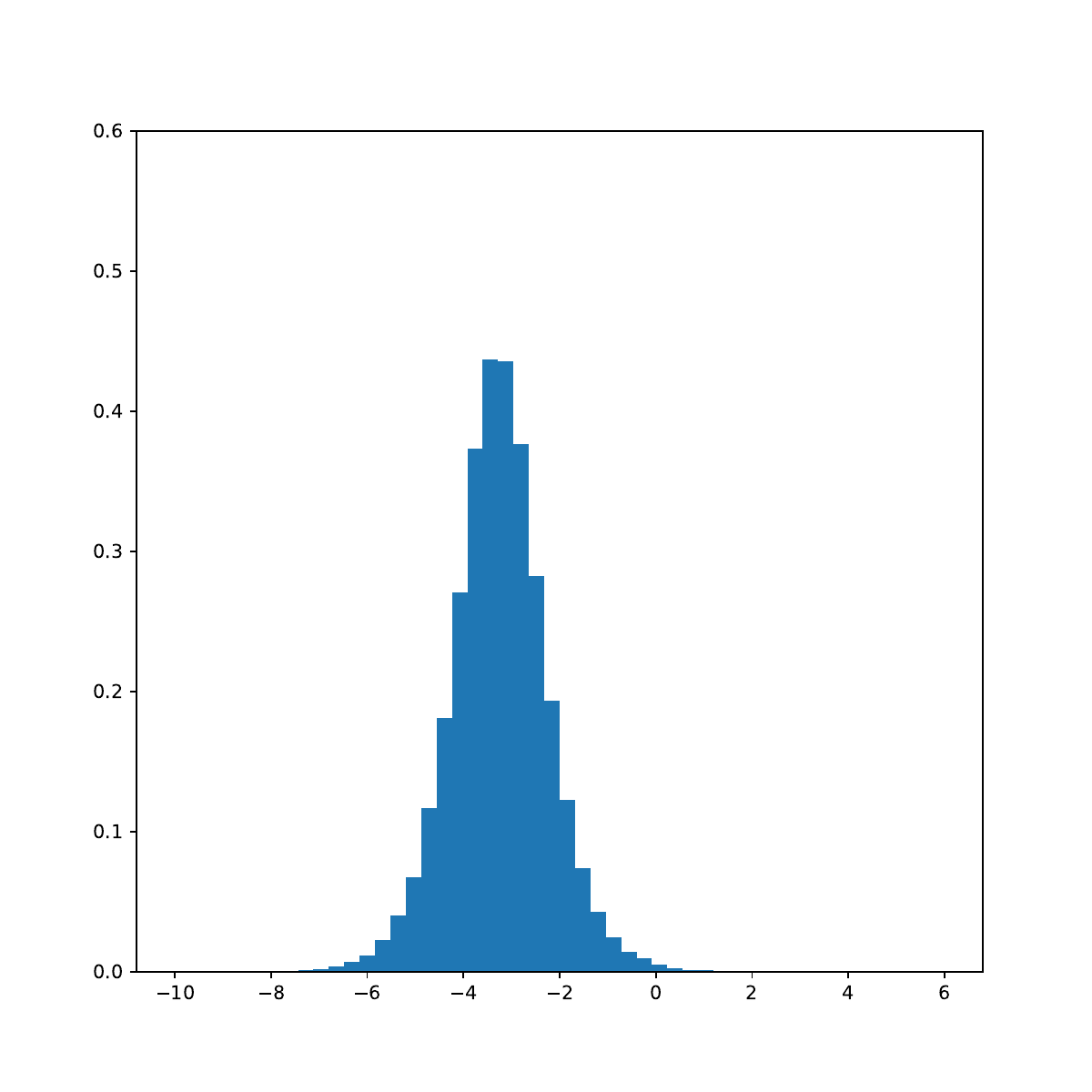}
        \caption{$t=1$}
    \end{subfigure}
    ~
    \begin{subfigure}{0.32\textwidth}
        \centering
        \includegraphics[width=0.95\linewidth]{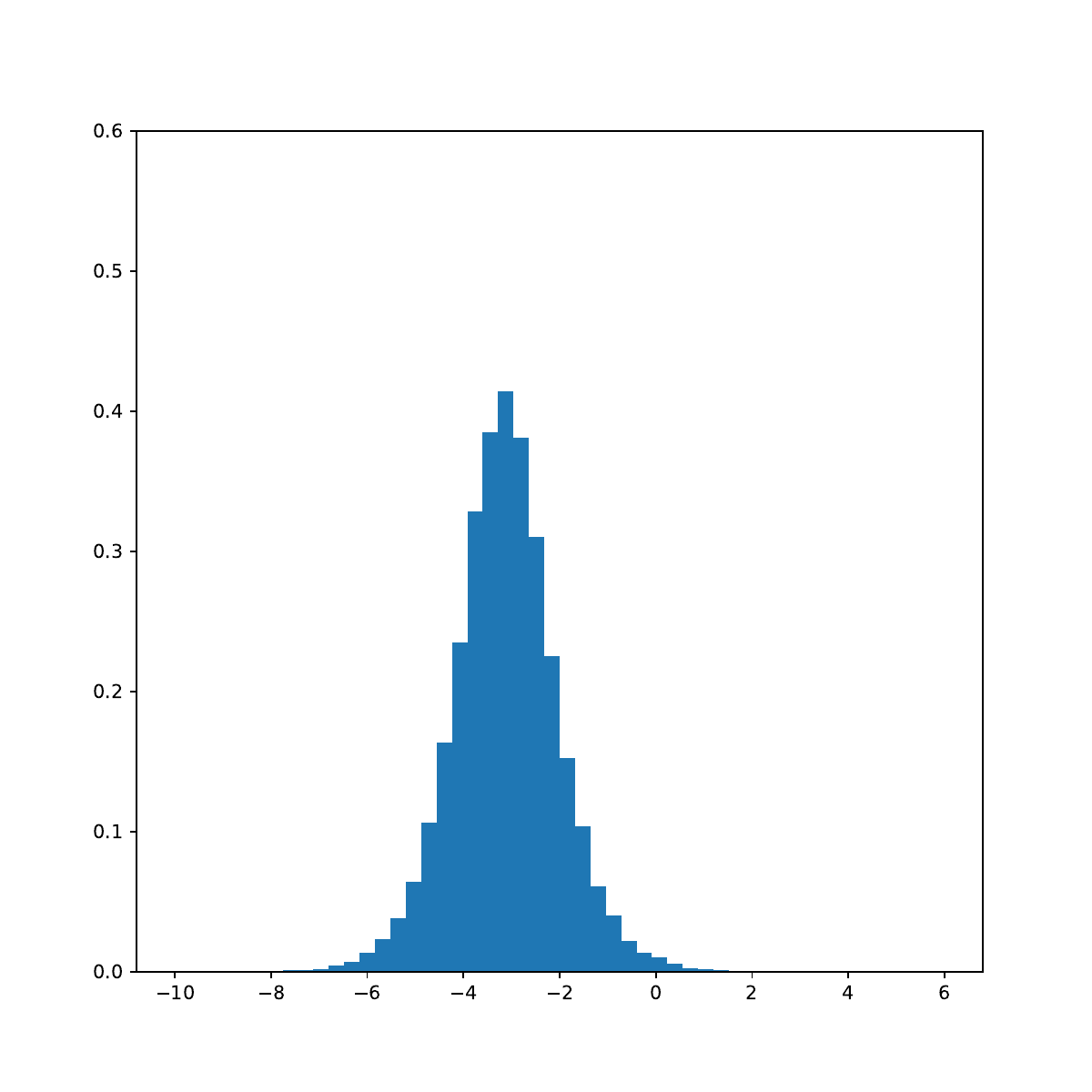}
        \caption{$t=1.25$}
    \end{subfigure}
    \caption{Sample histogram of computed $\rho_{\theta}$ at different time $t$ for $6$D GPE.}
    \label{6d GPE sampleplot}
\end{figure*}

\subsection{TDSE for system with three interactive-particles}
Finally, we present a numerical example of solving a system with three interactive-particles. We consider the three particles in a 3D space, with the origin exerting an attractive force on each of them.
%
We set the potential $\mathcal{F}_R$ to be
\begin{align}
    \mathcal{F}_R(\rho) = \iiint \left[-\sum_{k}\frac{3}{\epsilon + \left|z_k\right|} + \sum_{k<j}\frac{1}{\epsilon + \left|z_k-z_j\right|}\right]\rho(z_1, z_2, z_3)dz_1dz_2dz_3,
\end{align}
where $z_k := (x_{3k-2}, x_{3k-1}, x_{3k})^\top \in \mathbb{R}^3$ for $k=1, 2, 3$, $\epsilon $ is a small positive number which is set to $0.005$ in our computation, and $\left|\,\cdot\,\right|$ is the $l^2$ norm in Euclidean space. The $L^2$ first variation of $\mathcal{F}_R$ is:
\begin{align}
    \frac{\delta}{\delta \rho}\mathcal{F}_R(\rho) = -\sum_{k}\frac{3}{\epsilon + \left|z_k\right|} + \sum_{k<j}\frac{1}{\epsilon + \left|z_k-z_j\right|}.
\end{align}
The corresponding TDSE in the wavefunction form is:
\begin{align}
    \label{def: Li TDSE}
    i\frac{\partial}{\partial t}\psi (t, x)=-\frac{1}{2}\Delta \psi(t, x) + \left[-\sum_{k}\frac{3}{\epsilon + \left|z_k\right|} + \sum_{k<j}\frac{1}{\epsilon + \left|z_k-z_j\right|}\right]\cdot \psi(t, x).
\end{align}
We want to generate sample trajectories following the density function $\rho_t(x_1,\dots,x_9)$ corresponding to the solution of \eqref{def: Li TDSE}.

In our simulation, the initial density function is taken as a Gaussian in $\mathbb{R}^9$:
\begin{align}
    \rho_0(z_1, z_2, z_3)=\frac{1}{(\sqrt{2\pi})^9} \exp \Big(-\frac{1}{2}\left[|z_1-c_1|^2+|z_2-c_2|^2+|z_3-c_3|^2\right] \Big)
\end{align}
where $c_1=(1, 0, 0), c_2=(0, 1, 0), c_3=(0, 0, 1)$. We take the initial $\Phi_0$ as:
\begin{align}
    \Phi_0(x) = \frac{x_2^2-x_3^2}{(x_2^2+x_3^2)^{0.4}}+ \frac{x_6^2-x_4^2}{(x_4^2+x_6^2)^{0.4}}+\frac{x_7^2-x_8^2}{(x_7^2+x_8^2)^{0.4}}.
\end{align}
%

Again we use the same Neural-ODE architecture for GPE in Section \ref{sec: GPE}, consisting of two hidden layers, each containing 50 neurons, with the hyperbolic tangent as the activation function. 
We generate $12,000$ samples for the computation of the operator $G$ while $3,000,000$ samples are used to estimate the potential energy $\mathcal{F}$.
These choices are intuitively set according to the computation complexities of estimating $G$ and $F$, which are quadratic and linear with respect to the number of samples, respectively.
%
%
The time step size is set as $h=0.001$. We plot the histogram of $x_1$ versus time in Figure \ref{9d TDSE sampleplot}.
\begin{figure*}[t!]
    \begin{subfigure}{0.24\textwidth}
        \centering
        \includegraphics[width=0.95\linewidth]{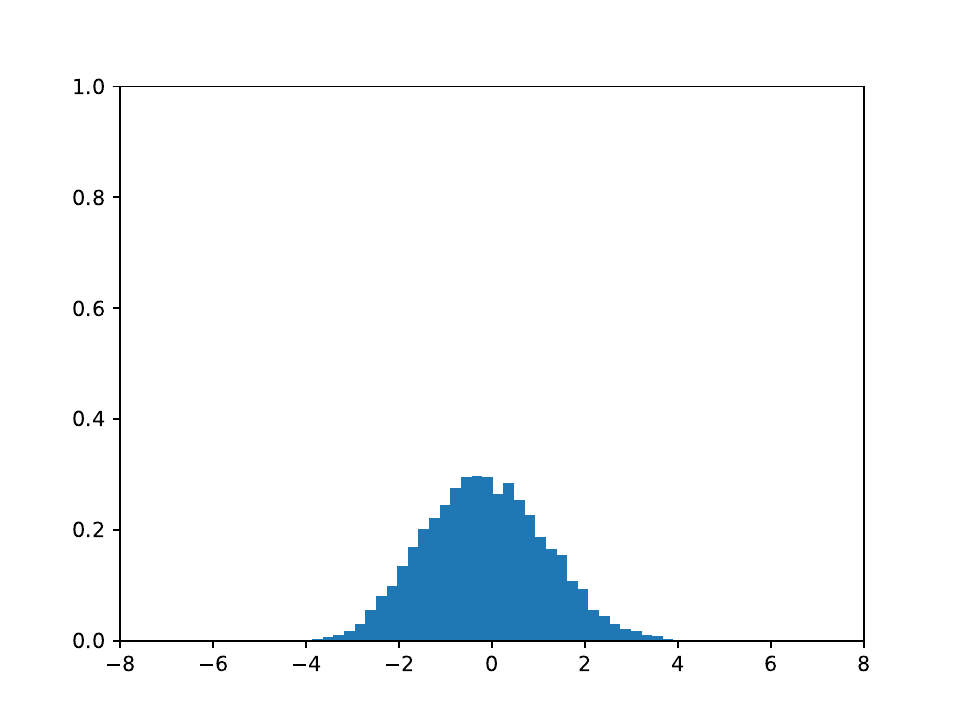}
        \caption{$t=0$}
    \end{subfigure}%
    ~
    \begin{subfigure}{0.24\textwidth}
        \centering
        \includegraphics[width=0.95\linewidth]{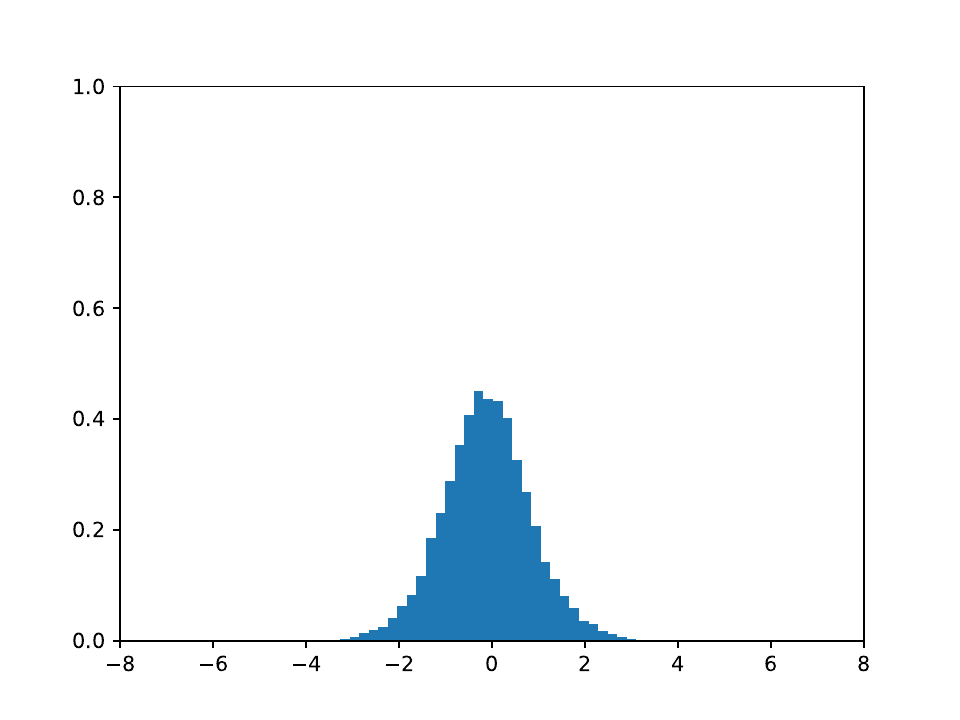}
        \caption{$t=0.3$}
    \end{subfigure}
    ~
    \begin{subfigure}{0.24\textwidth}
        \centering
        \includegraphics[width=0.95\linewidth]{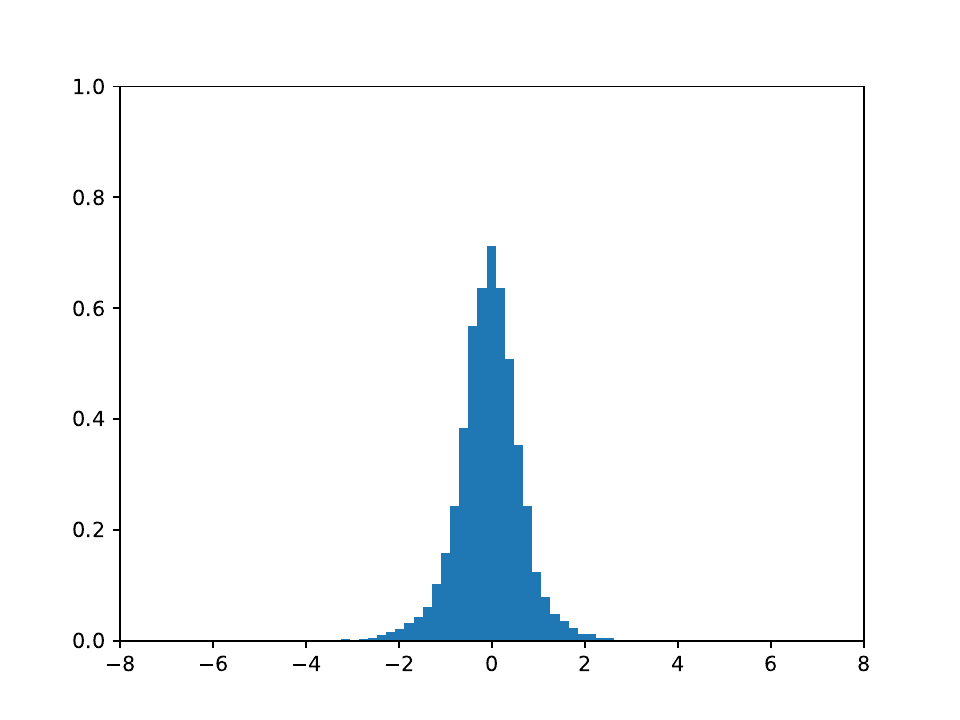}
        \caption{$t=0.6$}
    \end{subfigure}
    ~
    \begin{subfigure}{0.24\textwidth}
        \centering
        \includegraphics[width=0.95\linewidth]{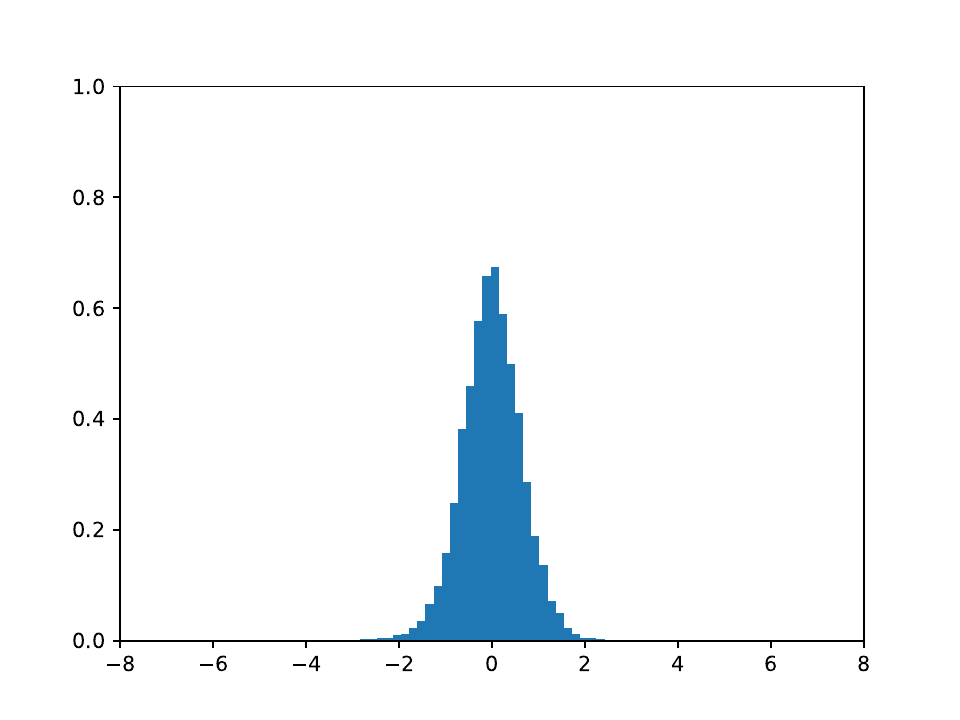}
        \caption{$t=0.9$}
    \end{subfigure}
    \\
    \begin{subfigure}{0.24\textwidth}
        \centering
        \includegraphics[width=0.95\linewidth]{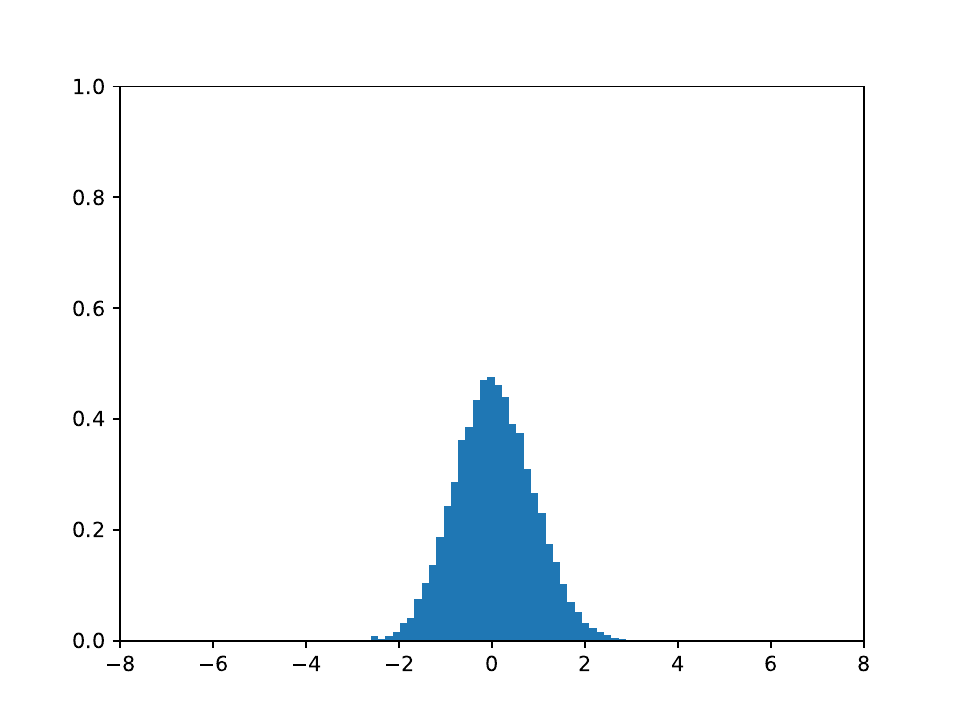}
        \caption{$t=1.2$}
    \end{subfigure}
    ~
    \begin{subfigure}{0.24\textwidth}
        \centering
        \includegraphics[width=0.95\linewidth]{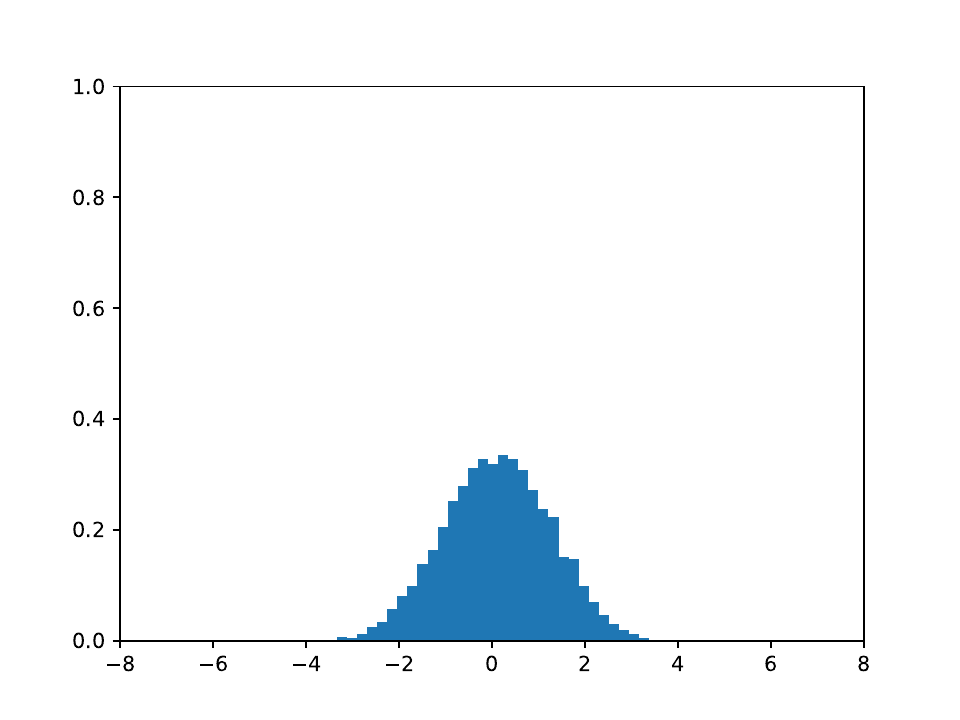}
        \caption{$t=1.5$}
    \end{subfigure}
    ~
    \begin{subfigure}{0.24\textwidth}
        \centering
        \includegraphics[width=0.95\linewidth]{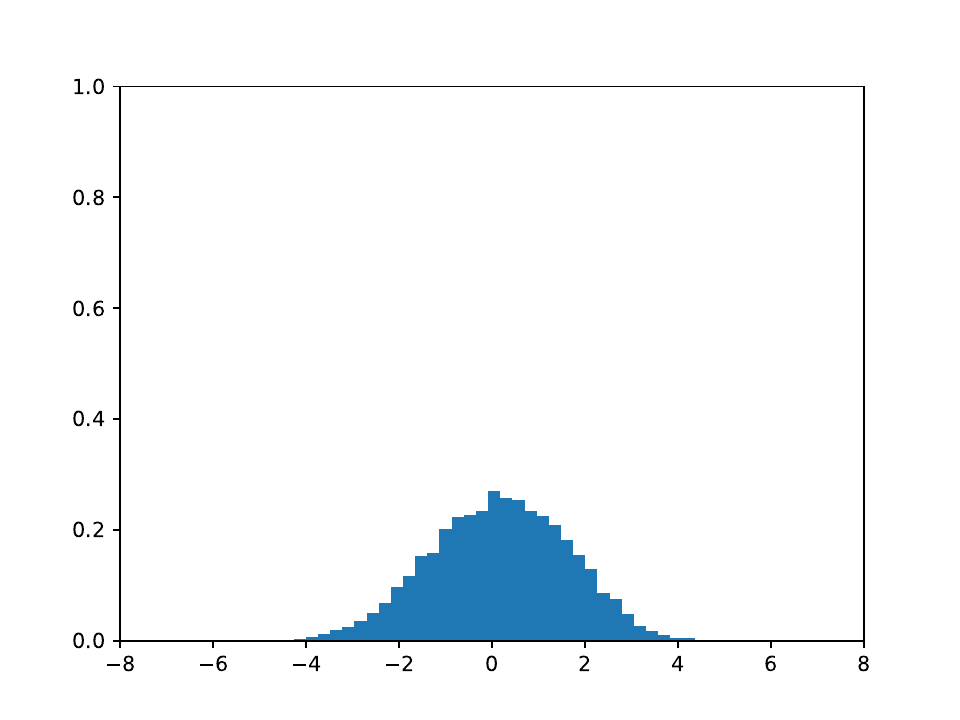}
        \caption{$t=1.8$}
    \end{subfigure}%
    ~
    \begin{subfigure}{0.24\textwidth}
        \centering
        \includegraphics[width=0.95\linewidth]{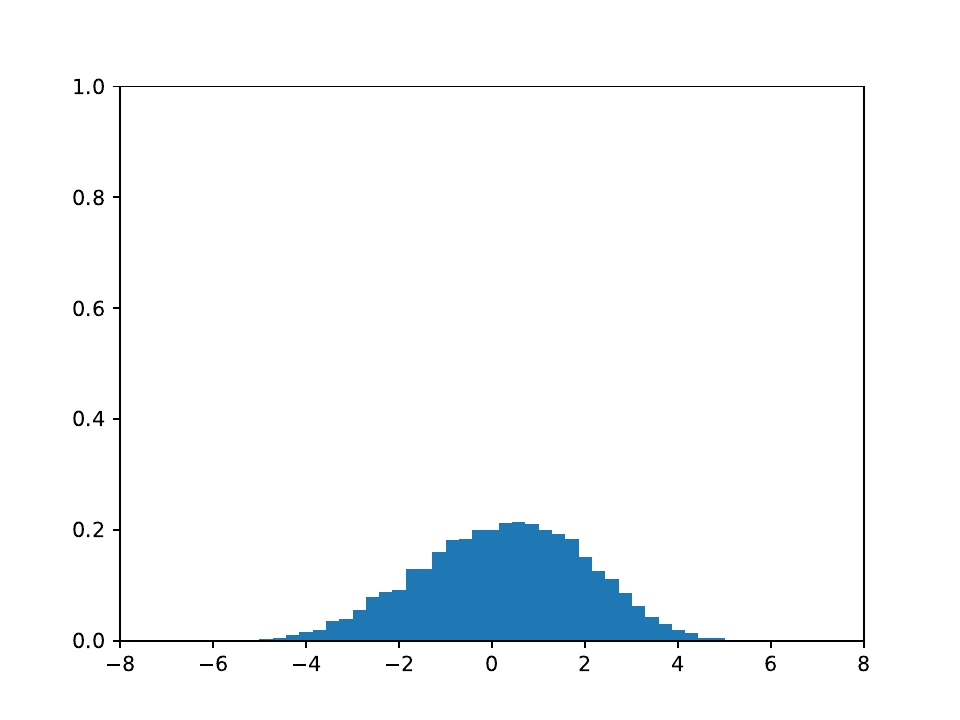}
        \caption{$t=2.1$}
    \end{subfigure}
    \\
    \begin{center}
    \begin{subfigure}{0.24\textwidth}
        \centering
        \includegraphics[width=0.95\linewidth]{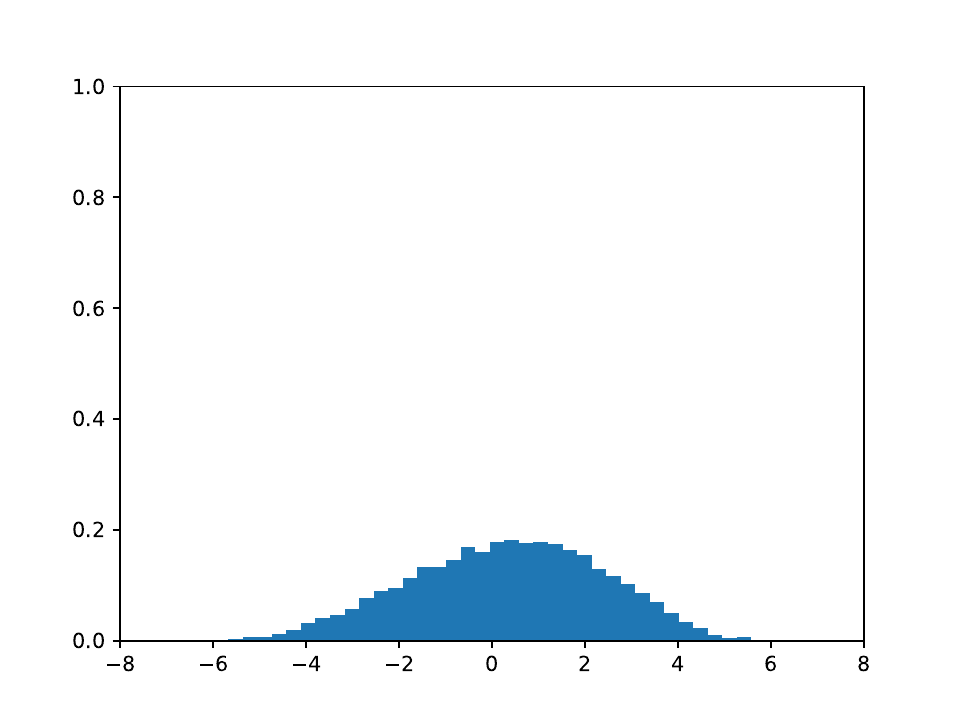}
        \caption{$t=2.4$}
    \end{subfigure}
    ~
    \begin{subfigure}{0.24\textwidth}
        \centering
        \includegraphics[width=0.95\linewidth]{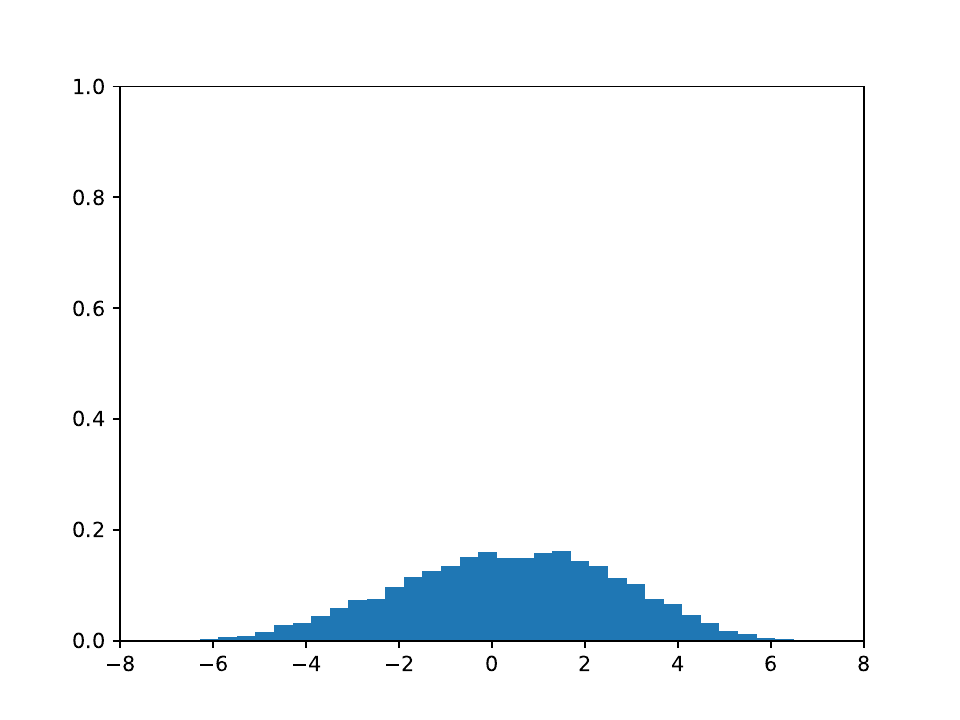}
        \caption{$t=2.7$}
    \end{subfigure}
    ~
    \begin{subfigure}{0.24\textwidth}
        \centering
        \includegraphics[width=0.95\linewidth]{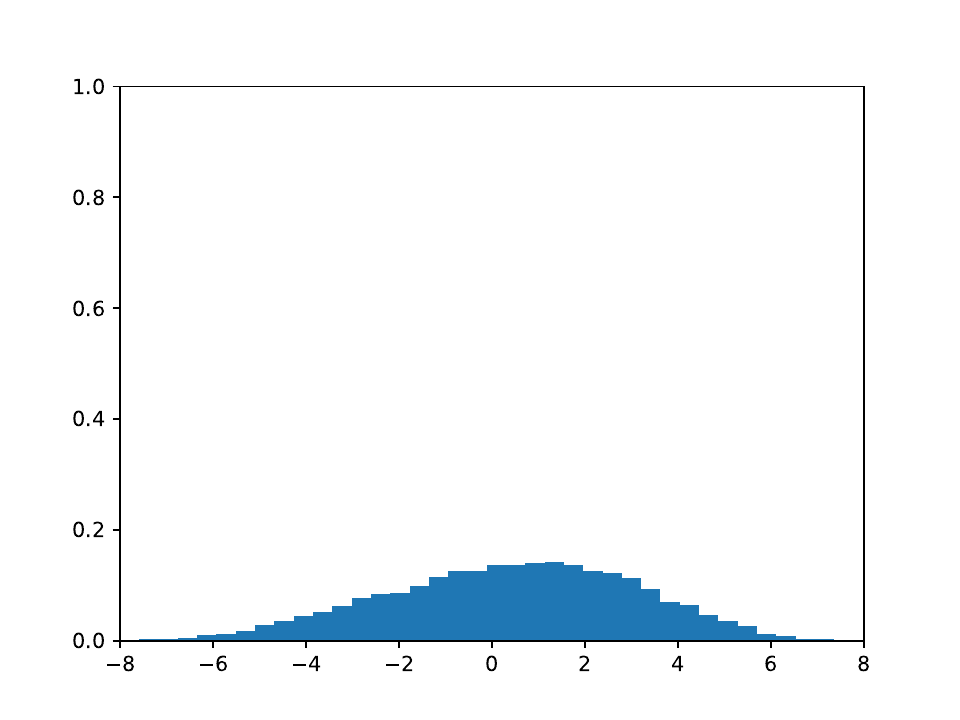}
        \caption{$t=3$}
    \end{subfigure}
    \end{center}   
    \caption{Sample histogram of computed $\rho_{\theta}$ at different time $t$ for $9$D TDSE. }
    \label{9d TDSE sampleplot}
\end{figure*}

In Figure \ref{9d TDSE trajectory plot}, we plot the trajectories of an example triple $(z_1(t),z_2(t),z_3(t))$ where $z_k(t) \in \mathbb{R}^3$ for $k=1,2,3$ and $(z_1(0),z_2(0),z_3(0))$ is randomly sampled from the initial distribution. 
\begin{figure*}[t!]
    \begin{subfigure}{0.24\textwidth}
        \centering
        \includegraphics[width=0.95\linewidth]{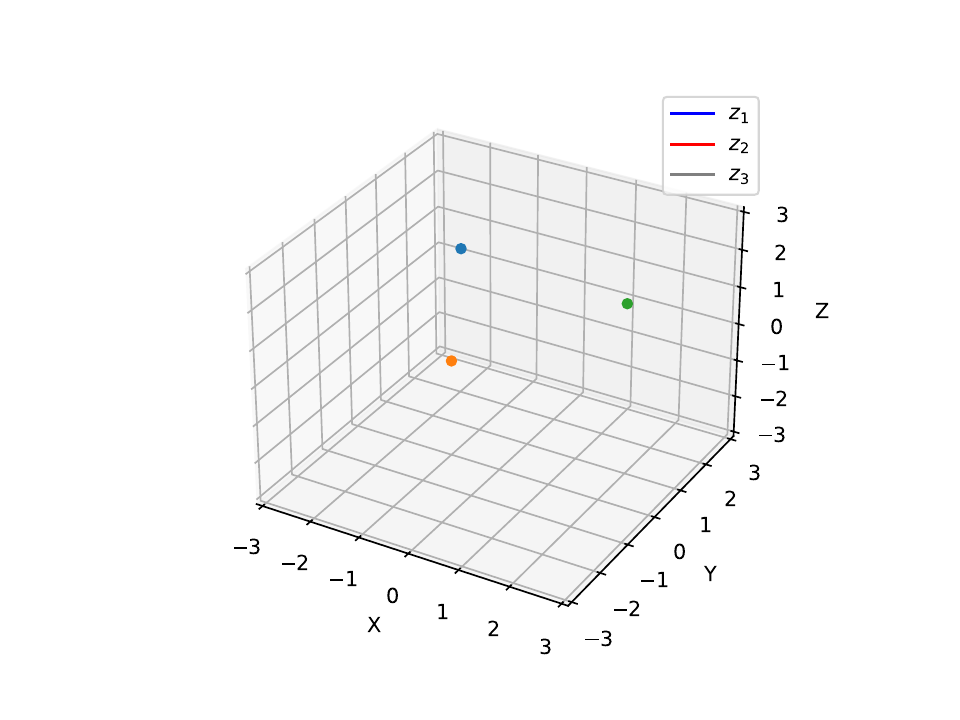}
        \caption{$t=0$}
    \end{subfigure}%
    ~
    \begin{subfigure}{0.24\textwidth}
        \centering
        \includegraphics[width=0.95\linewidth]{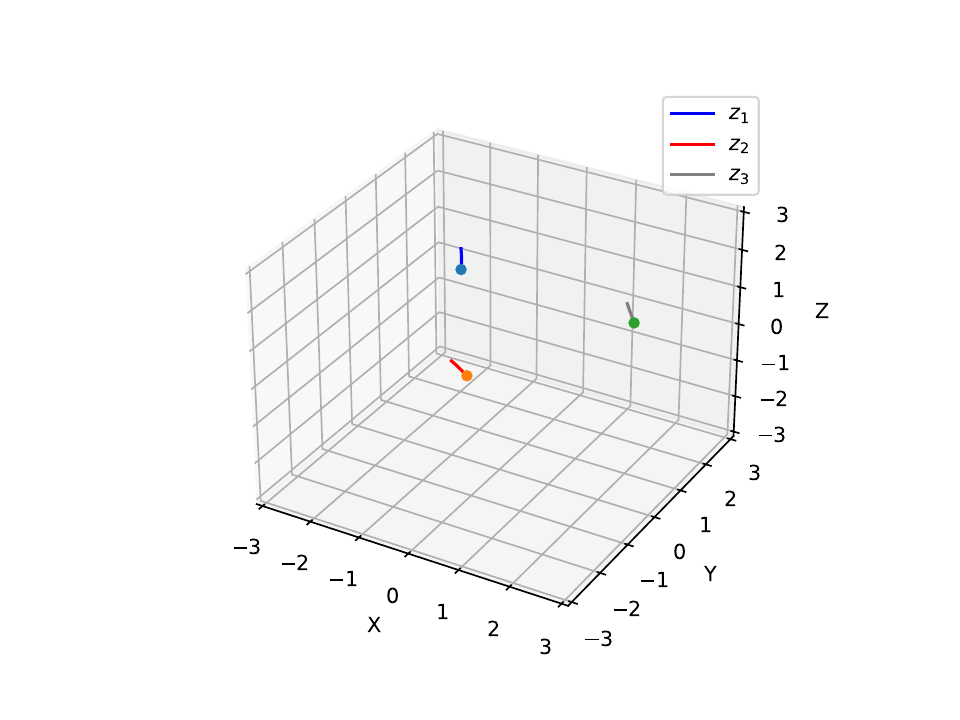}
        \caption{$t=0.3$}
    \end{subfigure}
    ~
    \begin{subfigure}{0.24\textwidth}
        \centering
        \includegraphics[width=0.95\linewidth]{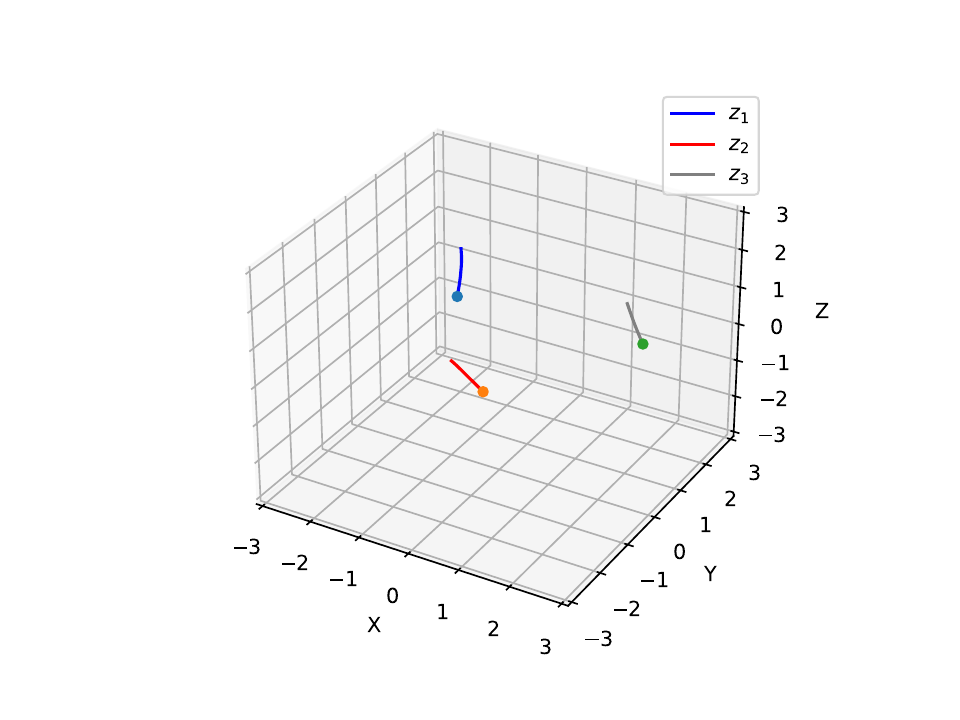}
        \caption{$t=0.6$}
    \end{subfigure}
    ~
    \begin{subfigure}{0.24\textwidth}
        \centering
        \includegraphics[width=0.95\linewidth]{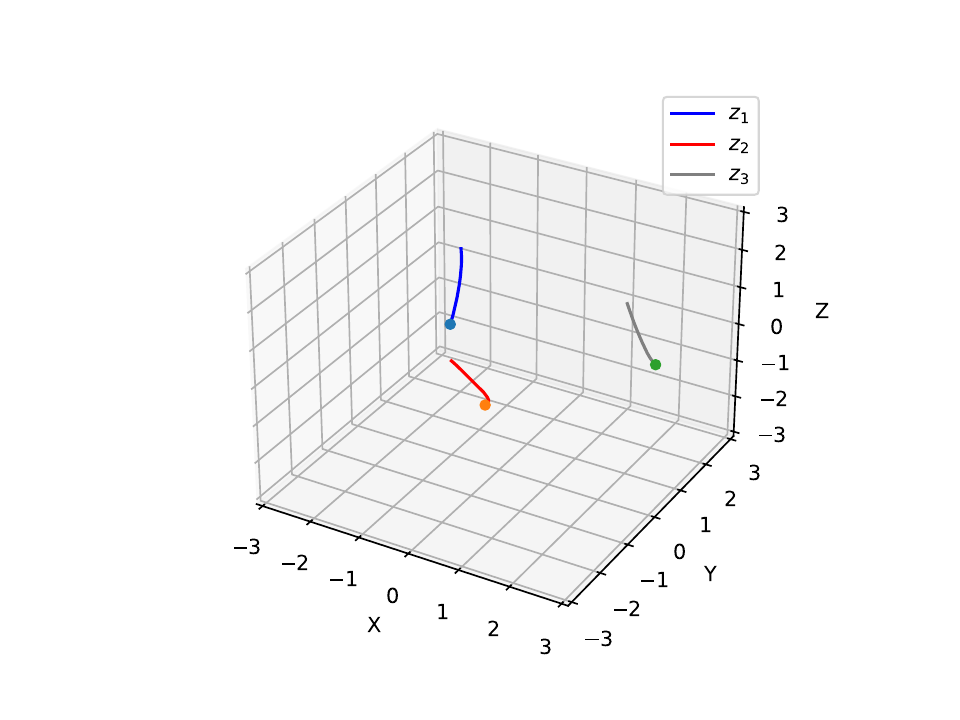}
        \caption{$t=0.9$}
    \end{subfigure}
    \\
    \begin{subfigure}{0.24\textwidth}
        \centering
        \includegraphics[width=0.95\linewidth]{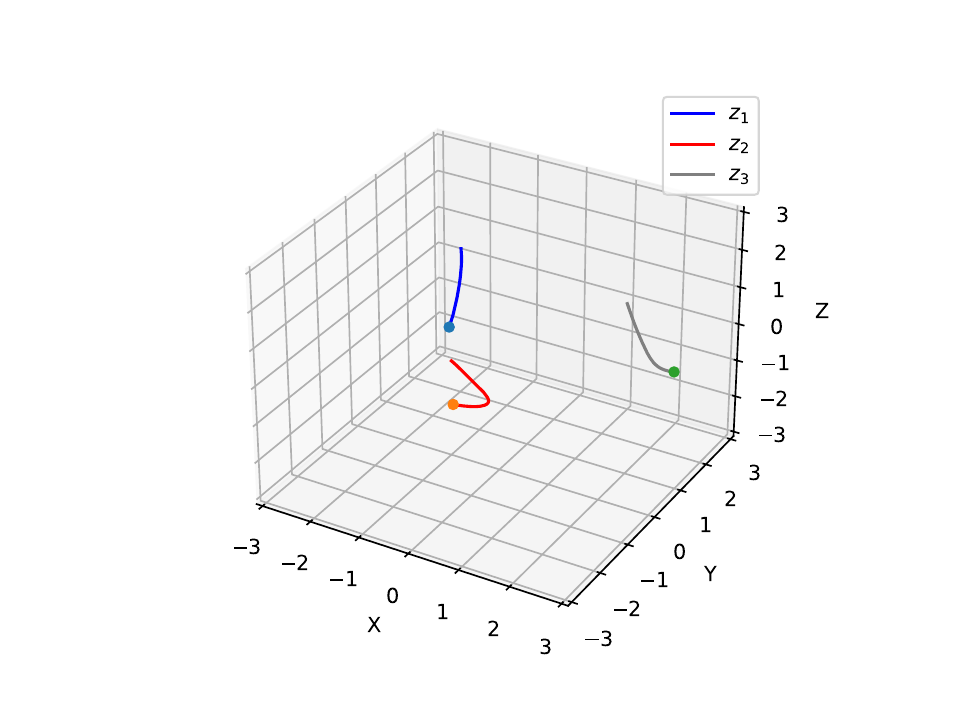}
        \caption{$t=1.2$}
    \end{subfigure}
    ~
    \begin{subfigure}{0.24\textwidth}
        \centering
        \includegraphics[width=0.95\linewidth]{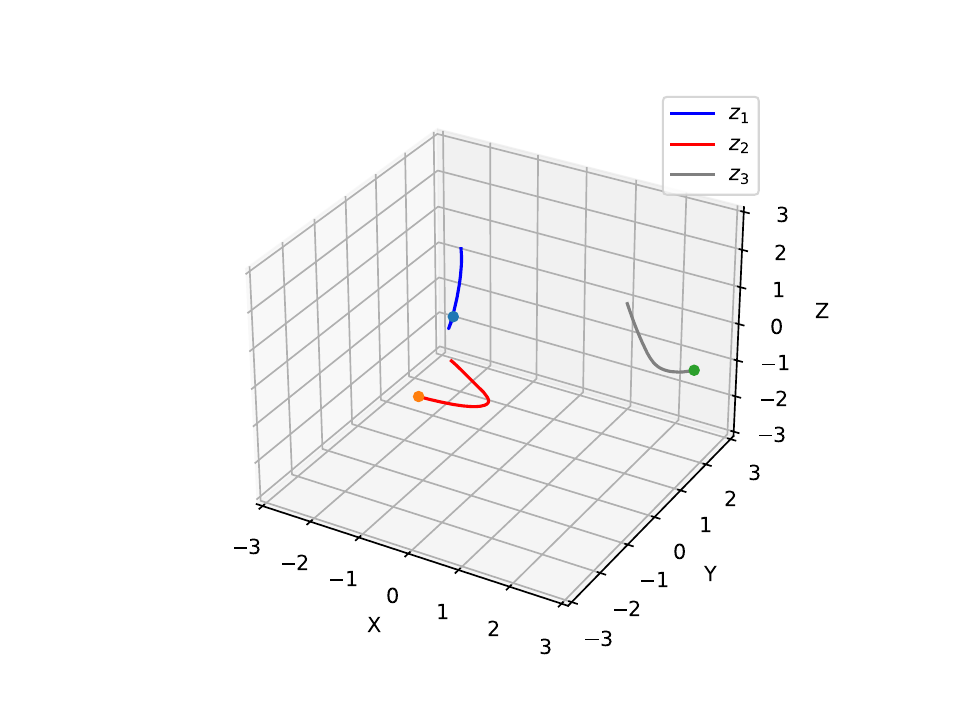}
        \caption{$t=1.5$}
    \end{subfigure}
    ~
    \begin{subfigure}{0.24\textwidth}
        \centering
        \includegraphics[width=0.95\linewidth]{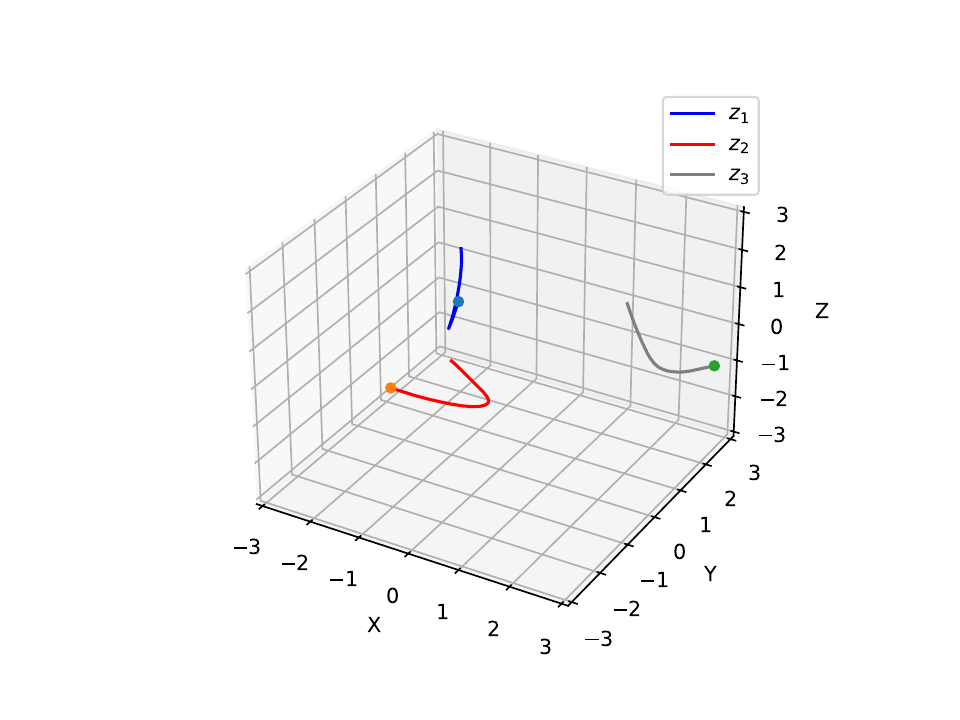}
        \caption{$t=1.8$}
    \end{subfigure}%
    ~
    \begin{subfigure}{0.24\textwidth}
        \centering
        \includegraphics[width=0.95\linewidth]{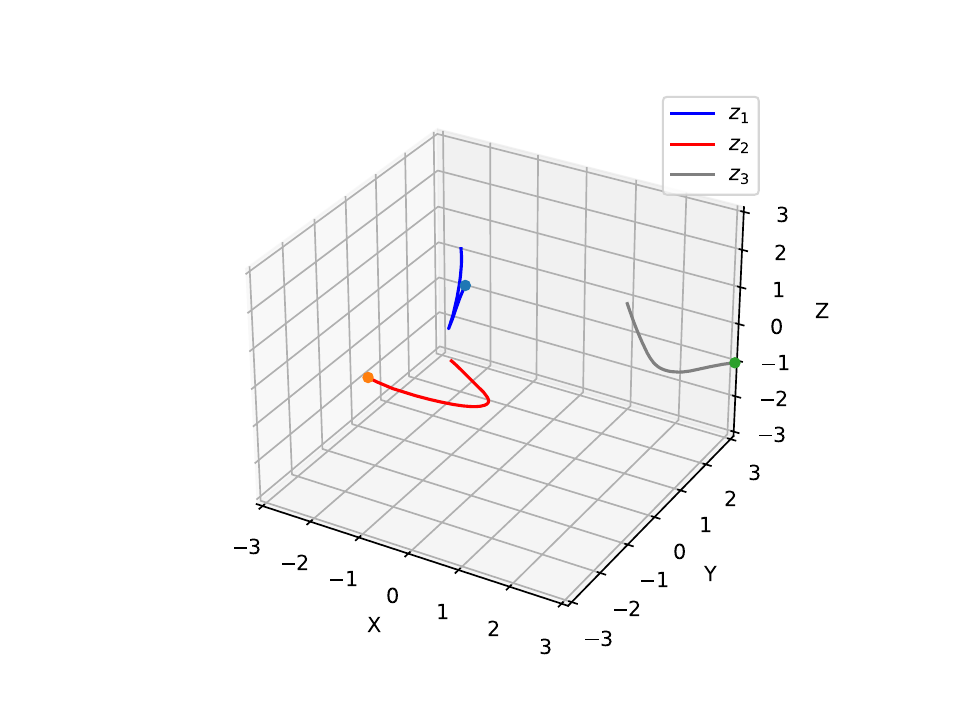}
        \caption{$t=2.1$}
    \end{subfigure}
    
    \caption{Trajectories of an example triple $(z_1(t),z_2(t),z_3(t)) \in \mathbb{R}^9$ with $z_k(t)=(x_{3k-2}(t),x_{3k-1}(t),x_{3k}(t)) \in \mathbb{R}^3$ for $k=1,2,3$ at different time slots $t$ for the interactive system experiment which is a $9$D TDSE. }
    \label{9d TDSE trajectory plot}
\end{figure*}
We also plot in Figure \ref{fig: hamiltonian 1 9d} the time evolution of the kinetic energy (orange), interactive potential $\mathcal{F}_R(\rho)$ (blue), and the Hamiltonian (green) in the simulation. The algorithm preserves the Hamiltonian up to time $t=3$ in this 9D example.

\begin{figure}
    \centering
    \includegraphics[width=0.5\linewidth]{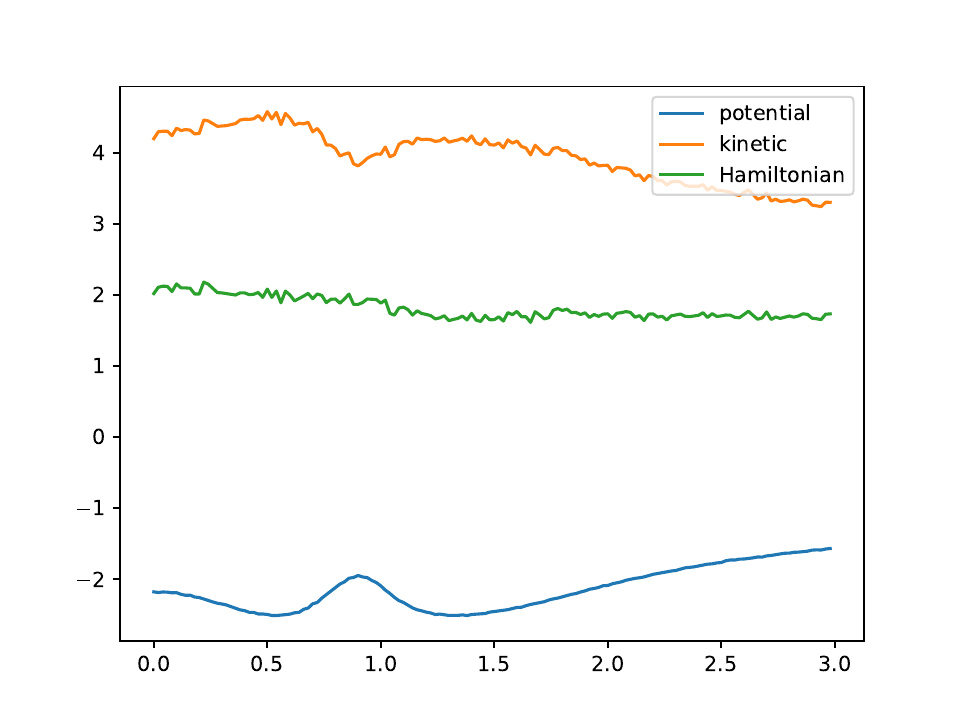}
    \caption{Evolution of the kinetic energy, interactive potential, and Hamiltonian in the numerical simulation for the 9D interactive system example.}
    \label{fig: hamiltonian 1 9d}
\end{figure}

\section{Conclusion} We propose a reformulation of TDSE \eqref{def: TDSE} in terms of push-forward map inspired by strategies developed for the generative model. This is accomplished by viewing TDSE as a Wasserstein Hamiltonian flow in the probability density manifold via the Madelung transform. Unlike the traditional wavefunction formulation of TDSE or Bohmian mechanics, the new formulation describes the dynamics in the space of diffeomorphisms. A main benefit is its convenience in handling the density evolution and the particle dynamics in a single closed system. Using the properties of Neural-ODE, we derived the corresponding equations in the parameter space of the neural network and designed an algorithm to simulate the solutions numerically. The algorithm provides an alternative to the existing methods. Since it works on samples and neural network parameter space, it is computationally friendly to problems in high dimensions. We demonstrated its performance in various problems including a 9D particle system simulation on a desktop. It is also worth mentioning that the simulations are done by the traditional ODE solvers, and there is no data needed to train the neural network. 

Meanwhile, the reformulation of the TDSE in the generative model invites a number of interesting questions. For example, how do we theoretically understand and analyze the dynamics in terms of the push forward map? Can we design new neural network structures to capture the dynamics more efficiently? What is the dependence of computational accuracy on the network structure (including the depth, width, and activation functions) and the sample size used in the calculation of $G$ and $\mathcal{F}$? Can the error be bounded theoretically in Wasserstein metric? Those are among a long list of questions that are worth exploring in future studies. 

\section{Acknowledgments} This research is partially supported by NSF grants DMS-1925263, DMS-2152960, DMS-2307465, DMS-2307466 and DMS-2409868. All authors made equal contributions. The authors would like to thank Aleksei Ustimenko for his valuable suggestions to this manuscript. 

\bibliographystyle{siamplain}
\bibliography{references}

\appendix

\section{Evaluation of Fisher information}\label{sec: evaluation of fi} 

We denote $x=T_{\theta}(z)$ and consider evaluating the parameterized Fisher information:
\begin{align}
    \label{def: parameterized fi}
    F_Q(\theta)=\mathcal{F}_Q(\rho_{\theta})=\frac{1}{8}\int_{\mathbb{R}^d}\lvert \nabla_x \log\ \rho_{\theta}(x)\rvert^2\rho_{\theta}(x)dx=\frac{1}{8}\int_{\mathbb{R}^d}\lvert \nabla_x \log\ \rho_{\theta}(T_{\theta}(z))\rvert^2\lambda(z)dz.
\end{align}
The gradient of log density term needs careful treatment, since $\log \rho_{\theta}\circ T_{\theta}(\cdot)$ is a function of $z$ while the gradient is taken with respect to $x$. The following algorithm is designed to compute \eqref{def: parameterized fi} efficiently with PyTorch:
\begin{algorithm}[H]
\caption{Compute Fisher information}
\label{alg: compute FI}
\begin{algorithmic}\STATE{Define a reference density $\lambda$, and a reversible neural network $T_{\theta}(z)$. 
}
\STATE{Generate samples $\{z_1, \cdots, z_N\}$ from $\lambda$, and compute $x_i=T_{\theta}(z_i)$.}
\STATE{Detach $\{x_i\}_{i=1}^N$ from computational graph and set $\{x_i\}_{i=1}^N$ as leaf variable by enabling requires\_gradient property.}
\STATE{Construct a function $h(x)=\log \rho_{\theta}\circ T_{\theta}\circ T^{-1}_{\theta}(x)$. $h(x)$ is a function of $x=T_{\theta}(z)$ with traced gradient computation.  } 
\STATE{Take the gradient of $\sum_{i=1}^N h(x_i)$ with respect to $\{x_i\}_{i=1}^N$, which gives $\{\nabla _{x_i}h(x_i)=\nabla _{x_i}\log\ \rho_{\theta}(x_i)\}_{i=1}^N$.}
\STATE{Compute the empirical Fisher information $\widehat{F}_Q(\theta)=\frac{1}{8N}\sum_{i=1}^N\lvert \nabla _{x_i}h(x_i)\rvert ^2$.}
\STATE{Compute $\nabla _{\theta}\widehat{F}_Q(\theta)$ using backpropagation.}
\STATE{\textbf{Output:} empirical Fisher information $\widehat{F}_Q(\theta)$ and its gradient $\nabla _{\theta}\widehat{F}_Q(\theta)$.}
\end{algorithmic}
\end{algorithm}


In our experiments, Neural-ODE is used as the push-forward map, since it provides an efficient evaluation of $T^{-1}_{\theta}$ and $\log \rho_{\theta}\circ T_{\theta}(\cdot)$. $T^{-1}_{\theta}$ can be constructed simply by reversing the time evolution of the ODE, and $\log \rho_{\theta}\circ T_{\theta}(\cdot)$ can be computed using the instant change of variables formula in \cite{chen2018neural}. 
\end{document}

%% file: ex_shared.tex
\usepackage{lipsum}
\usepackage{amsfonts}
\usepackage{graphicx}
\usepackage{epstopdf}
\usepackage{algorithmic}
\usepackage{subcaption}
\usepackage{mathtools}
\usepackage{amsmath}
\ifpdf
  \DeclareGraphicsExtensions{.eps,.pdf,.png,.jpg}
\else
  \DeclareGraphicsExtensions{.eps}
\fi

\newsiamremark{remark}{Remark}
\newsiamremark{hypothesis}{Hypothesis}
\crefname{hypothesis}{Hypothesis}{Hypotheses}
\newsiamthm{claim}{Claim}

\headers{Parameterized Schr\"odinger Equation}{H. Wu, S. Liu, X. Ye, and H. Zhou}

\title{A parameterized Wasserstein Hamiltonian flow approach for solving the Schr\"odinger equation
}

\author{Hao Wu\thanks{Charlotte, NC, USA
 (\email{hwu406@gmail.com}).}
 \and
 Shu Liu\thanks{Department of Mathematics, University of California, Los Angles, CA, USA(\email{shuliu@math.ucla.edu}).}
 \and 
 Xiaojing Ye\thanks{Department of Mathematics and Statistics, Georgia State University, Atlanta, GA, USA(\email{xye@gsu.edu}).}
 \and
 Haomin Zhou\thanks{School of Mathematics, Georgia Institute of Technology, Atlanta, GA, USA(\email{hmzhou@gatech.edu}).}
 }

\usepackage{amsopn}

\makeatletter
\newcommand*{\addFileDependency}[1]{
  \typeout{(#1)}
  \@addtofilelist{#1}
  \IfFileExists{#1}{}{\typeout{No file #1.}}
}
\makeatother

\newcommand*{\myexternaldocument}[1]{%
    \externaldocument{#1}%
    \addFileDependency{#1.tex}%
    \addFileDependency{#1.aux}%
}